\documentclass[11pt]{amsart}

\usepackage{a4wide}
\usepackage[sort,nospace,compress,noadjust]{cite}
\usepackage[shortlabels]{enumitem}
\usepackage[hidelinks]{hyperref}
\usepackage{mathtools}

\usepackage{amsmath, amssymb, amsthm}
\usepackage{mathrsfs}

\usepackage{xspace}
\usepackage{float}

\usepackage[foot]{amsaddr}

\newcommand{\faktor}[2]{{#1}/{#2}}

\usepackage{tikz}
\usetikzlibrary{positioning,fit,calc,shapes}
\usepackage{tikz-cd}
\usepackage{tkz-graph}

\numberwithin{equation}{section}

\theoremstyle{plain}
\newtheorem{THM}{Theorem}[section]
\newtheorem{PROP}[THM]{Proposition}
\newtheorem{LEM}[THM]{Lemma}
\newtheorem{COR}[THM]{Corollary}
\newtheorem{CLM}{Claim}
\newtheorem*{nCLM}{Claim}
\theoremstyle{definition}
\newtheorem{DEF}[THM]{Definition}
\newtheorem{EX}[THM]{Example}
\newtheorem{OBS}[THM]{Observation}
\theoremstyle{remark}
\newtheorem{REM}[THM]{Remark}

\renewcommand{\leq}{\leqslant}
\renewcommand{\geq}{\geqslant}
\renewcommand{\phi}{\varphi}
\renewcommand{\epsilon}{\varepsilon}

\newcommand{\bA}{\mathbf{A}}
\newcommand{\bB}{\mathbf{B}}
\newcommand{\bC}{\mathbf{C}}
\newcommand{\bD}{\mathbf{D}}
\newcommand{\bF}{\mathbf{F}}
\newcommand{\bH}{\mathbf{H}}
\newcommand{\bM}{\mathbf{M}}
\newcommand{\bN}{\mathbf{N}}
\newcommand{\bU}{\mathbf{U}}
\newcommand{\bV}{\mathbf{V}}
\newcommand{\bX}{\mathbf{X}}
\newcommand{\bY}{\mathbf{Y}}
\newcommand{\bZ}{\mathbf{Z}}

\newcommand{\sss}{\mathbf{s}}
\newcommand{\ttt}{\mathbf{t}}
\newcommand{\xx}{\mathbf{x}}
\newcommand{\yy}{\mathbf{y}}

\newcommand{\NN}{\mathbb{N}}
\newcommand{\QQ}{\mathbb{Q}}
\newcommand{\RR}{\mathbb{R}}
\newcommand{\ZZ}{\mathbb{Z}}

\newcommand{\restr}{\mathord{\upharpoonright}}
\newcommand{\embedsto}{\hookrightarrow}
\newcommand{\id}{\mathrm{id}}
\newcommand{\Age}{\operatorname{Age}}
\newcommand{\ar}{\operatorname{ar}}
\newcommand{\AP}{\operatorname{AP}}
\newcommand{\HP}{\operatorname{HP}}
\newcommand{\JEP}{\operatorname{JEP}}
\newcommand{\Sym}{\operatorname{Sym}}

\newcommand{\calC}{\mathcal{C}}
\newcommand{\calF}{\mathcal{F}}
\newcommand{\calK}{\mathcal{K}}
\newcommand{\calM}{\mathcal{M}}
\newcommand{\calN}{\mathcal{N}}
\newcommand{\calT}{\mathcal{T}}

\newcommand{\Tc}{\calT_C}
\newcommand{\Tcc}{\calT_{C'}}
\newcommand{\Tf}{\calT_\bF}

\newcommand{\Fraisse}{Fra\"\i{}ss\'e\xspace}
\newcommand{\Hubicka}{Hubi{\v{c}}ka\xspace}
\newcommand{\Nesetril}{Ne{\v{s}}et{\v{r}}il\xspace}
\newcommand{\Todorcevic}{Todor\v{c}evi\'c\xspace}
\newcommand{\Katetov}{Kat\v{e}tov\xspace}
\newcommand{\Masulovic}{Ma\v{s}ulovi\'c\xspace}
\newcommand{\Kubis}{Kubi\'s\xspace}
\newcommand{\Frechet}{Fr\'echet\xspace}
\newcommand{\Chvatal}{Chv\'a{}tal\xspace}

\DeclareMathOperator{\Aut}{Aut}
\DeclareMathOperator{\dom}{dom}
\DeclareMathOperator{\im}{im}
\DeclareMathOperator{\sgn}{sgn}
\DeclareMathOperator{\Hom}{Hom}
\DeclareMathOperator{\Iso}{Iso}
\DeclareMathOperator{\Lip}{Lip}
\DeclareMathOperator{\ob}{ob}
\DeclareMathOperator{\Mid}{Mid}

\newcommand{\dM}{d_M}
\newcommand{\dN}{d_N}

\newcommand{\cat}{\mathscr}
\newcommand{\Id}{\operatorname{Id}}
\newcommand{\dotcup}{\mathrel{\dot\cup}}

\newcommand{\ie}{i.e.\@\xspace}
\newcommand{\eg}{e.g.\@\xspace}

\title{Echeloned Spaces}
 %, 06A05}
\author[M.\,Gheysens]{Maxime Gheysens$^1$}
\address{$^1$Institute of Discrete Mathematics and Algebra\\
        Faculty of Mathematics and Computer Science\\
        Technische Universit\" at Bergakademie Freiberg\\
        D-09596 Freiberg\\
        Germany}
\email{maxime.gheysens@normalesup.org}
\thanks{The research of the first and the fifth author was supported by the German Academic Exchange Service (DAAD) through Project-ID 57602973}

\author[B.\,Pavlica]{Bojana Pavlica$^2$}
\address{$^2$Department of Mathematics and Informatics\\
        University of Novi Sad\\
        21000 Novi Sad\\
        Serbia} 
\email{bojana@dmi.uns.ac.rs}
\thanks{The research of the second and the fourth author was supported by the Ministry of  Science, Technological Development and Innovation of the Republic of Serbia through Grant No. 451-03-783/2021-09/3}

\author[Ch.\,Pech]{Christian Pech$^3$}
\address{$^3$Institute of Mathematics\\
        Czech Academy of Sciences\\
        \v{Z}itn\'a 25\\ 
        115\,67 Praha 1\\ 
        Czech Republic}
\email{pech@math.cas.cz}
\thanks{The research of the third author was supported by GA~\v{C}R (Czech Science Foundation) grant EXPRO 20-31529X}

\author[M.\,Pech]{Maja Pech$^{2\ast}$} 
\email{maja@dmi.uns.ac.rs}
\thanks{The fourth author was supported by the Ministry of  Science, Technological Development and Innovation of the Republic of Serbia through Grant No. 451-03-47/2023-01/200125}
 
\author[F.\,M.\,Schneider]{Friedrich Martin Schneider$^1$}
\email{Martin.Schneider@math.tu-freiberg.de}

\thanks{$^\ast$ corresponding author}

\begin{document}

\begin{abstract}
    We introduce the notion of echeloned spaces --- an order-theoretic abstraction of metric spaces. The first step is to characterize metrizable echeloned spaces. It turns out that morphisms between metrizable echeloned spaces are uniformly continuous or have a uniformly discrete image. In particular, every automorphism of a metrizable echeloned space is uniformly continuous, and for every metric space  with midpoints the automorphisms of the induced echeloned space are precisely the dilations. 

    Next we focus on finite echeloned spaces. They form a \Fraisse class and we describe its \Fraisse-limit both as the echeloned space induced by a certain homogeneous metric space and as the result of a random construction. Building on this we show that the class of finite ordered echeloned spaces is Ramsey. The proof of this result combines a combinatorial argument by \Nesetril and \Hubicka with a topological-dynamical point of view due to Kechris, Pestov, and \Todorcevic. Finally, using the method of \Katetov functors due to \Kubis and \Masulovic, we prove that the full symmetric group on a countable set topologically embeds into the automorphism group of the countable universal homogeneous echeloned space. 
\end{abstract}

\maketitle
	The notion of a metric is ubiquitous in mathematics. This is no doubt due to its great versatility for capturing information of topological, geometrical, and order-theoretical nature. 
	
	The first abstract definitions of metric spaces given by  \Frechet (see \cite[p.772]{Fre05}, \cite[p.18]{Fre06}) and Hausdorff (see \cite[p.211]{Hau14}) had as a goal to capture  the notion of convergence. This lead to the notion of a topology or of a uniformity induced by a metric and is certainly the main focus of research concerning metric spaces. 
	
	Geometrical aspects of abstract metric spaces were first treated by Menger (see \cite{Men28}). He studied convexity of metric spaces, and he examined  the basic properties of  betweenness relations. A complete axiomatization of betweenness relations  in metric spaces was given only recently by \Chvatal (see \cite{Chv04}).
	
	As an example of an order-theoretic aspect of metric spaces we mention the phenomenon of boundedness. The notion of abstract boundedness (now better known as bornology) was introduced and studied by Hu (see \cite{Hu49}). In particular, Hu characterized metrizable bornologies.  

    Last but not least, every metric space gives rise to a (bounded) coarse structure in the sense of Roe (see~\cite{Roe03}). Roughly speaking, such a coarse structure captures geometric properties of the space on a large scale (\ie  up to a uniformly bounded error).
  
	All the structures mentioned above have one thing in common: they are definable from metric spaces while the actual numerical distance between two points is of no great importance.  For instance, a metric  may be scaled by a positive real number without making any difference concerning convergence, boundedness, or convexity. For topological considerations only the very small distances are of interest, for coarse geometries large distances are relevant, and for geometrical considerations mainly qualitative properties like collinearity stand in the focus. Finally, in bornological considerations the order relation between distances is crucial. 
	
	Motivated by these observations, in this paper we introduce \emph{echeloned spaces}. These are spaces in which the closeness between pairs of points cannot be measured but only compared. 	Echeloned spaces appear to capture very well the order-theoretic aspects of metric spaces. 
	
		It should be mentioned that a notion similar to echeloned spaces was suggested by Pestov when discussing nearest neighbor classifiers in machine learning \cite[Observa\c{c}\~ao 5.4.40]{Pes19}. In contrast to our approach, Pestov compares distances of points to a given point. 
	
	In Section~\ref{sec_original} after the basic definitions we settle the question of metrizability of echeloned spaces (see Proposition~\ref{MScharacter}).

    Section~\ref{metrizableEchelonedSpaces} is concerned with morphisms between metrizable echeloned spaces. The main result of this section is a characterization of the automorphisms of echeloned spaces  induced by metric spaces with midpoints (see Proposition~\ref{similarity}).

    Section~\ref{sec_Fraisse} contains the proof of the existence of a countable universal homogeneous echeloned space $\bF$ (using \Fraisse's Theorem). It is shown that this space is not the echeloned space induced by the countable universal homogeneous rational metric space, a.k.a.\ the rational Urysohn space (see Corollary~\ref{notUry}). We proceed to showing that the edge-coloured graph induced by $\bF$ is in fact universal and homogeneous as an edge-coloured graph (see Theorem~\ref{prop_F-1isC-colHomUniGraph}), and we give a probabilistic construction of this graph (see Proposition~\ref{randConst}).

    In Section~\ref{sec_Ramsey} we show that the class of finite ordered echeloned spaces has the Ramsey property in the sense of \cite{Nes05} (see Theorem~\ref{thm_ShasRamsey}). The proof combines a combinatorial result by \Hubicka and \Nesetril  \cite{hubivcka2019all} with the Kechris-Pestov-\Todorcevic correspondence \cite{kechris2003fraisse}.

    In Section~\ref{sec_Katetov} it is shown that the category of finite echeloned spaces with embeddings may be endowed with a \Katetov functor in the sense of \cite{KubMas16}. As a direct consequence we obtain that the automorphism group of the countable universal homogeneous echeloned space contains the full symmetric group on a countable set as a closed topological subgroup.

    Throughout the paper we use standard model-theoretic notation and notions, see \cite{hodges1993model}.

\section{Echeloned spaces}\label{sec_original}

We define echeloned spaces as structures whose pairs of points are comparable. 

\begin{DEF}\label{orgDEF_sim} Let $X$ be a non-empty set. Then, a pair $\bX=(X,\leq_\bX)$ is called an \emph{echeloned space} if $(X^2,\leq_\bX)$ is a prechain\footnote{A pair $(C,\leq)$ is called a \emph{prechain} if $\leq$ is a total preorder on a set $C$, \ie $\leq$ is a reflexive, transitive, and linear binary relation on $C$.} satisfying
\begin{enumerate}[label=(\roman*), ref=(\roman*)]
\item for all $x, y, z \in X$: $(x,x) \leq_\bX (y, z)$,
\item for all $x, y, z \in X$ if $(y, z) \leq_\bX (x,x)$ then $y = z$, and
\item for all $x, y \in X$: $(x, y) \leq_\bX (y,x)$.
\end{enumerate}
The relation $\leq_\bX$ is called an \emph{echelon} on $X$. 

Given an echelon $\leq_\bX$ on a set $X$, we introduce $\sim_\bX\, \subseteq X^2$ as follows
\[(x_1, y_1) \sim_\bX (x_2, y_2) \quad :\Longleftrightarrow \quad (x_1, y_1) \leq_\bX (x_2, y_2) \quad \textrm{ and } \quad (x_2, y_2) \leq_\bX (x_1, y_1).\]
\end{DEF}

\begin{REM}\label{rem_echelon}
Formally, one could have introduced echeloned spaces as relational structures over a signature $\{\lesssim\}$, where $\ar(\lesssim) = 4$. Then, an echeloned space $(X, \leq_\bX)$ would in fact be a $\{\lesssim\}$-structure $\bX = (X, \lesssim_\bX)$ where $(x_1, y_1, x_2, y_2) \in{} \lesssim_\bX$ if and only if $(x_1, y_1) \leq_\bX (x_2, y_2)$. 
This translation suggests a natural definition for homomorphisms and embeddings between echeloned spaces: if $\bX$ and $\bY$ are two echeloned spaces, then a map $f\colon X\to Y$ is going to be called a homomorphism (embedding) from $\bX$ to $\bY$ if and only if it is a homomorphism (embedding) between their corresponding $\{\lesssim\}$-structures $(X,\lesssim_\bX)$ and $(Y,\lesssim_\bY)$. 
\end{REM}

Clearly, for any echeloned space $\bX = (X, \leq_\bX)$, the relation  $\sim_\bX$ is an equivalence relation on $X^2$. The echelon $\leq_\bX$ naturally induces a linear ordering on the quotient set $\faktor{X^2}{\sim_\bX}$, written in symbols as $\leq_{E(\bX)}$, as follows:
\[ 
\quad [(x_1, x_2)]_{\sim_\bX} \leq_{E(\bX)} [(y_1, y_2)]_{\sim_\bX} \quad :\Longleftrightarrow \quad (x_1, x_2) \leq_\bX (y_1, y_2). 
\]
We shall refer to $E(\bX) \coloneqq  \left(\faktor{X^2}{\sim_\bX},{} \leq_{E(\bX)}\right)$ as the \emph{echeloning} of $\bX$. Lastly, let $\eta_{\bX} \colon X^2 \twoheadrightarrow E(\bX)$ be the quotient map.

\begin{LEM}\label{lem_sim_HOM}
Let $\bX$ and $\bY$ be two echeloned spaces. Then, a map $h \colon X \to Y$ is a  homomorphism from $\bX$ to $\bY$ if and only if there exists a (necessarily unique) homomorphism of ordered sets $\hat{h} \colon E(\bX) \to E(\bY)$ for which $\hat{h} \circ \eta_\bX = \eta_\bY \circ h^2$, \ie the diagram below commutes:
\[\begin{tikzcd}
X^2 \arrow[r, twoheadrightarrow, "\eta_\bX"]\arrow[d, rightarrow, "h^2"{name=L, left}] & E(\bX)\arrow[d, rightarrow, "\hat{h}"] \\
Y^2 \arrow[r, twoheadrightarrow, "\eta_\bY"]& E(\bY).
\end{tikzcd}
\]
\end{LEM}

\begin{proof}
``$\Rightarrow$'': First, assume that $h \colon X \to Y$ is a homomorphism between the echeloned spaces $\bX$ and $\bY$. Define $\hat{h} \colon \faktor{X^2}{\sim_\bX} \to \faktor{Y^2}{\sim_\bY},\, [(x_1, x_2)]_{\sim_\bX} \mapsto [(h(x_1), h(x_2))]_{\sim_\bY}$. First, we show that it is well-defined. Let $x_1, x_2, x_1', x_2' \in X$ and $(x_1, x_2) \sim_\bX (x_1', x_2')$, \ie $[(x_1, x_2)]_{\sim_\bX} = [(x_1', x_2')]_{\sim_\bX}$. Then, $(h(x_1), h(x_2)) \sim_\bY (h(x_1'), h(x_2'))$ since $h$ preserves $\leq_\bX$. Consequently, \[\hat{h}([(x_1, x_2)]_{\sim_\bX}) = [(h(x_1), h(x_2))]_{\sim_\bY} = [(h(x_1'), h(x_2'))]_{\sim_\bY} = \hat{h}([(x_1', x_2')]_{\sim_\bX}).\]
Next, we prove that $\hat{h}$ preserves $\leq_{E(\bX)}$. Take any $x_1, x_2, x_1', x_2' \in X$ such that $[(x_1, x_2)]_{\sim_\bX} \leq_{E(\bX)} [(x_1', x_2')]_{\sim_\bX}$. By definition, this means that $(x_1, x_2) \leq_\bX (x_1', x_2')$. Consequently, given that $h$ is a homomorphism it holds that $(h(x_1), h(x_2)) \leq_\bY (h(x_1'), h(x_2'))$. So, 
 \[\hat{h}([(x_1, x_2)]_{\sim_\bX}) = [(h(x_1), h(x_2))]_{\sim_\bY} \leq_{E(\bY)} [(h(x_1'), h(x_2'))]_{\sim_\bY} = \hat{h}([(x_1', x_2')]_{\sim_\bX}).\]
Finally, we show that for our choice of $\hat{h}$ the above diagram commutes. Take any $x_1, x_2 \in X$. Then, 
\[\hat{h} \circ \eta_\bX (x_1, x_2) = \hat{h} ([(x_1, x_2)]_{\sim_\bX}) = [(h(x_1), h(x_2))]_{\sim_\bY} = \eta_\bY(h(x_1), h(x_2)) = \eta_\bY \circ h^2 (x_1, x_2), \]
which is what we needed to show.

``$\Leftarrow$'': Assume $\hat{h} \colon E(\bX) \to E(\bY)$ is a homomorphism and $h \colon X \to Y$ a map such that $\hat{h} \circ \eta_\bX = \eta_\bY \circ h^2$. We prove $h$ is a homomorphism from $\bX$ to $\bY$. Take any $x_1, x_2,x_1', x_2' \in X$ such that $(x_1, x_2) \leq_\bX (x_1', x_2')$. This is equivalent to saying that 
\[ \eta_\bX(x_1, x_2) = [(x_1, x_2)]_{\sim_\bX} \leq_{E(\bX)} [(x_1', x_2')]_{\sim_\bX} = \eta_\bX(x_1', x_2'), \]
by definition of $\leq_{E(\bX)}$. It follows that:
\[ \hat{h}\circ \eta_\bX (x_1, x_2) = \hat{h} ([(x_1, x_2)]_{\sim_\bX}) \leq_{E(\bY)} \hat{h}([(x_1', x_2')]_{\sim_\bX}) = \hat{h}\circ \eta_\bX (x_1', x_2').\]
Now, from the commutativity of the given diagram, that is equivalent to:
\[\eta_\bY(h(x_1), h(x_2)) =  \eta_\bY \circ h^2 (x_1, x_2)  \leq_{E(\bY)} \eta_\bY \circ h^2 (x_1', x_2') = \eta_\bY(h(x_1'), h(x_2')).\]
In other words, $[(h(x_1), h(x_2))]_{\sim_\bY} \leq_{E(\bY)} [(h(x_1'), h(x_2'))]_{\sim_\bY}$ or better yet $(h(x_1), h(x_2)) \leq_\bY (h(x_1'), h(x_2'))$, so $h$ does preserve $\leq_\bX$ and is thus a homomorphism. 
\end{proof}

\begin{COR}\label{cor_sim_EMB}
Let $\bX$ and $\bY$ be two echeloned spaces. Then, $h \colon \bX \to \bY$ is an embedding if and only if there exists an embedding of ordered sets $\hat{h} \colon E(\bX) \embedsto E(\bY)$ for which $\hat{h} \circ \eta_\bX = \eta_\bY \circ h^2$, \ie the diagram below commutes:
\[\begin{tikzcd}
X^2 \arrow[r, twoheadrightarrow, "\eta_\bX"]\arrow[d, rightarrow, hook, "h^2"{name=L, left}] & E(\bX)\arrow[d, rightarrow, hook, "\hat{h}"] \\
Y^2 \arrow[r, twoheadrightarrow, "\eta_\bY"]& E(\bY).
\end{tikzcd}
\]    
\end{COR}

\begin{proof}
``$\Rightarrow$'': Assume $h$ is an embedding from $\bX$ to $\bY$. Then by Lemma~\ref{lem_sim_HOM}, there exists such a homomorphism $\hat{h} \colon E(\bX) \to E(\bY)$ for which $\hat{h} \circ \eta_\bX = \eta_\bY \circ h^2$. In order to prove that it does not only preserve, but also reflects $\leq_{E(\bX)}$, take any $x_1, x_2, x_1', x_2' \in X$ such that 
$\hat{h}([(x_1, x_2)]_{\sim_\bX}) \leq_{E(\bY)} \hat{h}([(x_1', x_2')]_{\sim_\bX})$. Put differently, 
\[ \eta_\bY \circ h^2 (x_1, x_2) = \hat{h} \circ \eta_\bX (x_1, x_2) \leq_{E(\bY)} \hat{h} \circ \eta_\bX (x_1', x_2') =  \eta_\bY \circ h^2 (x_1', x_2'), \]
\ie  $[(h(x_1), h(x_2))]_{\sim_\bY} \leq_{E(\bY)} [(h(x_1'), h(x_2'))]_{\sim_\bY}$. By definition of $\leq_{E(\bY)}$ we actually get $(h(x_1), h(x_2)) \leq_\bY (h(x_1'), h(x_2'))$. Finally, as $h$ reflects $\leq_\bX$ then $(x_1, x_2) \leq_\bX (x_1', x_2')$ which is what we wanted. What remains is to show that $\hat{h}$ is injective. Take any $x_1, x_2, x_1', x_2' \in X$ such that $\hat{h}([(x_1, x_2)]_{\sim_\bX}) = \hat{h}([(x_1', x_2')]_{\sim_\bX})$. Similarly as before, by the commutativity of the diagram, we get that $(h(x_1), h(x_2)) = (h(x_1'), h(x_2'))$. As $h$ itself is injective it follows immediately that $(x_1, x_2) = (x_1', x_2')$.

``$\Leftarrow$'': Let $h \colon X \to Y$ be a map and assume the existence of such an $\hat{h}$ as described in the statement of the right-hand side of the corollary. First, by Lemma~\ref{lem_sim_HOM} we get that $h$ is a homomorphism from $\bX$ to $\bY$. Then, take any $x_1, x_2, x_1', x_2' \in X$ for which $(h(x_1), h(x_2)) \leq_\bY (h(x_1'), h(x_2'))$. Notice how this leads to $[(h(x_1), h(x_2))]_{\sim_\bY} \leq_{E(\bY)} [(h(x_1'), h(x_2'))]_{\sim_\bY}$, which in turn translates to 
\[ \hat{h} \circ \eta_\bX (x_1, x_2) = \eta_\bY \circ h^2 (x_1, x_2) \leq_{E(\bY)} \eta_\bY \circ h^2 (x_1', x_2') = \hat{h} \circ \eta_\bX (x_1', x_2'), \]
\ie $[(x_1, x_2)]_{\sim_\bX} \leq_{E(\bX)} [(x_1', x_2')]_{\sim_\bX}$. Thus, $(x_1, x_2) \leq_\bX (x_1', x_2')$. Lastly, we show that $h$ is injective. Take any $x_1, x_2 \in X$, for which $h(x_1) = h(x_2)$. Then $(h(x_1), h(x_2)) = (h(x_2), h(x_2))$, and so 
\[ \hat{h}([(x_1, x_2)]_{\sim_\bX}) = \eta_\bY (h(x_1), h(x_2)) = \eta_\bY (h(x_2), h(x_2)) =  \hat{h}([(x_2, x_2)]_{\sim_\bX}). \]
As a result of $\hat{h}$ being injective we get that $(x_1, x_2) \sim_\bX (x_2, x_2)$. However, by the axioms of an echeloned space this leads to $x_1 = x_2$. This concludes the proof. 
\end{proof}

As the next lemma shows, every metric space induces an echeloned space on the same set of points. 

\begin{LEM}\label{MSinduceS}
Let $(M, \dM)$ be a metric space. Define a binary relation  $\leq_\bM$ on $M^2$ such that for every $(x_1, y_1), (x_2, y_2) \in M^2$:
\[(x_1, y_1) \leq_\bM (x_2, y_2) \quad \Longleftrightarrow \quad \dM(x_1, y_1) \leq \dM(x_2, y_2).\]
Then $(M, \leq_\bM)$ is an echeloned space.
\end{LEM}
\begin{proof}
We proceed to prove that $\leq_\bM$ is an echelon on $M$. As $\leq$ is a linear order on $\RR$, $\leq_\bM$ is trivially a prechain on $M^2$ by its definition. For any $x, y, z \in M$:
\begin{enumerate}[label=(\roman*), ref=(\roman*)]
\item $\dM(x,x) = 0 \leq \dM(y, z)$, thus $(x,x) \leq_\bM (y, z)$.
\item if $(y, z) \leq_\bM (x, x)$, then $\dM(y, z) \leq \dM(x, x) = 0$. So $\dM(y, z) = 0$, and so $y = z$.
\item $\dM(x, y) = \dM(y, x)$, so $(x, y) \leq_\bM (y, x)$.\qedhere
\end{enumerate}
\end{proof}

Obviously, not every echeloned space is induced by a metric space in this way. A necessary condition for $\bX$ to be induced by a metric space is that $E(\bX)$ embeds into the chain of reals, see Examples \ref{ex_ordinalproduct} and \ref{ex_omegaone}.

\begin{DEF}
A metric space $(X, d_X)$ is called \emph{dull} if 
$d_X(x,x') \leq d_X(y,y') + d_X(z, z')$ holds, for all $x, x', y, y', z, z' \in X$ with $y \not= y'$ and $z \not= z'$.
\end{DEF}
Recall that a metric space is called \emph{uniformly discrete} if all nonzero distances in it are bounded from below by some constant $c>0$. 
\begin{LEM}\label{lem_dullMS}
Any dull metric space is both bounded and uniformly discrete.
\end{LEM}

\begin{proof}
Let $(X, d_X)$ be a dull metric space. Define $\tau \coloneqq  \inf\{d_X(x, y) \mid x, y \in X, x \not= y\}$. Clearly, the diameter of $(X,d_X)$  is not larger than $2 \tau$, due to dullness. Thus, any distance within $(X, d_X)$ is bounded by $2 \tau$. 

Without loss of generality, assume $|X| \not= 1$. In other words, there exist such $x, y \in X$ that $d_X(x, y) \not= 0$. As $\tau \leq d_X(x, y) \leq 2 \tau$, it follows that $\tau > 0$.  Consequently, as $0 < \tau \leq d_X(x, y)$ for all distinct $x, y \in X$, $(X, d_X)$ is uniformly discrete.
\end{proof}

Recall that a subset $S$ of a prechain $(C,\leq)$ is called \emph{order-dense} in $C$ if and only if, for every $a < b$ in $C\setminus S$, there exists  $s\in S$, such that  $a < s < b$.

\begin{PROP}\label{MScharacter}
Let $(X,\leq_\bX)$ be an echeloned space. Then the following are equivalent:
\begin{enumerate}[label=(\roman*), ref=(\roman*)]
    \item $\leq_{\bX}$ is induced by a metric space,
    \item $\leq_{\bX}$ is induced by a dull metric space,
    \item $(X^2,\leq_{\bX})$ contains a countable order-dense subset.
\end{enumerate}
\end{PROP}
\begin{proof}
We show $(ii)\Rightarrow (i)\Rightarrow (iii)\Rightarrow (ii)$. It is clear that $(ii)\Rightarrow (i)$.

    ``$(i)\Rightarrow (iii)$'': Let $X$ be a set and $\leq_\bX$ an echelon on it induced by a metric space $(X, d_X)$. Clearly, $\im d_X$ is order-embeddable into the reals, given that $d_X$ is a metric. Thus, due to a classical result by Birkhoff (see \cite[Theorem VIII.24]{birkhoff1967lattice}), $\im d_X$ contains a countable order-dense subset. Let $T \subseteq X^2$ be a transversal of $\sim_\bX$. Observe that $(T, \leq_\bX) \cong (\im d_X, \leq)$.  Let $S \subseteq T$ be countable order-dense in $(T, \leq_\bX)$. Then it is easy to see that $S$ is also order-dense in $(X^2, \leq_\bX)$. 
    
   ``$(iii)\Rightarrow (ii)$'': Let $S$ be a countable order-dense subset of $X^2$.  
   Without loss of generality, we may assume that for any $(a,b), (c,d) \in S$: $(a,b) \not\sim_\bX (c,d)$. Additionally, we may expand $S$ to a transversal $T$ of  $\sim_\bX$. Clearly,  the set $T$, ordered by $\leq_\bX$,  is a chain that contains $S$ as a  countable order-dense subset. Moreover, $T$ contains a single identical pair as its smallest element.  By \cite[Theorem VIII.24]{birkhoff1967lattice}, there does exist an order embedding $f$ of $T$ into the reals. Without loss of generality, we may assume that the image of $f$ is contained in $(\{0\} \cup (1, 2), \leq)$, and that $f$ maps the identical pair to $0$.    
   
   In the following, for each pair  $(x,y)$ from $X^2$ let $(x,y)_T$ denote  the unique pair from $T$  that is in the same $\sim_\bX$-class as $(x,y)$.
   Now we define a map $d_X$ on $X^2$ as follows: for any $(x, y) \in X^2$ let $d_X(x, y) \coloneqq  f((x,y)_T)$. It remains to show that it is a well-defined metric.

   To begin with, notice that for any $x \in X$ the distance $d_X(x, x) = f ((x,x)_T) = 0$. For any $x, y \in X$, $d_X (x, y) \geq 0$. Also, since $(x, y) \sim_\bX (y, x)$ then $d_X (x, y) = d_X (y, x)$. At last, take any $x, y, z \in X$. 
   If $x$, $y$, and $z$ are pairwise distinct, then \[d_X(x, y) = f((x, y)_T) < 2 = 1 + 1 < f((x, z)_T) + f((z, y)_T) = d_X(x, z) + d_X(z, y).\] Otherwise, if $x=y$, then $d_X(x, y) = 0 \leq d_X(x, z) + d_X(z, y)$; whereas if $x=z$ or $y=z$, then the triangle inequality trivially holds.

   Overall, $(X, d_X)$ is indeed a metric space (inducing the echeloned space $(X, \leq_\bX)$). Moreover, by definition, it is dull.
\end{proof}

An echeloned space that is induced by a metric space will be called \emph{metrizable}. As a direct consequence of Proposition~\ref{MScharacter} we obtain:
\begin{COR}
Every echeloned space on a countable set is metrizable.
\end{COR}

What follows are two examples of non-metrizable echeloned spaces.

\begin{EX}\label{ex_ordinalproduct}
Observe the following chain $C = \RR_0^+ \odot \mathbf{2}$, where $\RR_0^+$ is the set of non-negative real numbers, and where $\mathbf{2} = \{0,1\}$ is the ordinal number 2. Recall that the \emph{lexicographic product} $X \odot Y$, of two disjoint posets $X$ and $Y$, is the set of all ordered pairs $(x, y)$ (where $x \in X, y \in Y$), ordered lexicographically, \ie by the rule that $(x,y) < (x', y')$ if and only if $x < x'$ or $x = x' \land y < y'$.  Clearly, the lexicographic product of any two chains is again a chain. Take any countable subset $S$ of $C$. As there exist uncountably many irrational numbers, there has to be some $a \in \RR_0^+ \setminus \QQ$ such that neither $(a, 0)$ nor $(a, 1)$ are in $S$. Due to the lexicographical ordering, there is no element of $C$, let alone of $S$, in between the formerly mentioned two. Therefore, $S$ cannot be order-dense in $C$. Since the choice of $S$ was arbitrary, by \cite[Theorem VIII.24]{birkhoff1967lattice}, $C$ is not embeddable into $\RR$.

Now, consider an echelon $\leq_\RR$ on the set of reals defined as follows:
\begin{align*}
(x,y) <_\RR (u,v) \quad :\Longleftrightarrow \quad  &\text{ either } \bigl||x|-|y|\bigr| < \bigl||u|-|v|\bigr|,\\ 
&\text{ or } \bigl||x|-|y|\bigr| = \bigl||u|-|v|\bigr|, \text{ but }\sgn(x) = \sgn(y),\, \sgn(u) \not= \sgn(v),
\end{align*}
for any $(x, y), (u, v) \in \RR^2$.   
Observe that the chain $(\faktor{\RR^2}{\sim_\RR}, \leq_\RR)$ is isomorphic to $(\RR^+_0 \odot \mathbf{2}, \leq)$. As a result of Proposition~\ref{MScharacter}, the latter echelon could not have been induced by a metric space.
\end{EX}

\begin{EX}\label{ex_omegaone}
Let $X \coloneqq  \{0,1\}^{\omega_1}$. Further, we define a map $f_\bX \colon X^2 \to \omega_1^+$ as follows. Let $\xx = (x_i)_{i \in \omega_1}$, $\yy = (y_i)_{i \in \omega_1} \in X$, then if $\xx \not= \yy$, we set $f_\bX(\xx,\yy) \coloneqq  k$, where $k$ is the smallest index for which $x_k \not= y_k$; otherwise $f_\bX(\xx,\yy) \coloneqq  \omega_1$. Moreover, we define a binary relation $\leq_\bX$ on $X^2$ so that
\[(\xx, \yy) \leq_\bX (\sss, \ttt) \quad :\Longleftrightarrow\quad f_\bX(\xx, \yy) \geq f_\bX(\sss, \ttt).\]
Clearly, $\leq_\bX$ is a well-defined echelon on $X$. Given that $f_\bX$ is surjective, the chain $\left(\faktor{X^2}{\sim_\bX}, \leq_\bX~\!\!\right)$ is isomorphic to $(\omega_1^+, \geq)$. As already $\omega_1$ does not order-embed into the reals\footnote{Assume for a contradiction that $f \colon \omega_1 \to \RR$ is an order-embedding. Then, for each $\alpha \in \omega_1$ choose a rational number $g(\alpha)$ such that $f(\alpha) < g(\alpha) < f(\alpha + 1)$. As a result $g \colon \omega_1 \to \QQ$ is an injection, which is impossible.}, neither does $\omega_1^+$. By \cite[Theorem VIII.24]{birkhoff1967lattice}\footnote{During this study it was realised that the proof given in the reference was not correct. The claim, however, still holds, and the corrected proof will appear in the PhD Thesis of the second author} $(\omega_1^+, \geq)$ does not have a countable order-dense subset. Consequently, by Proposition~\ref{MScharacter} the echelon $\leq_\bX$ is not induced by a metric space.
\end{EX}

We have already provided a full characterisation of metrizable echeloned spaces, cf. Proposition~\ref{MScharacter}. 

For any two metric spaces $\calM = (M, \dM)$ and $\calN = (N, \dN)$, we define $\Hom(\calM, \calN)$ as the set of all $1$-Lipschitz maps from $\calM$ to $\calN$. Note that not every homomorphism between metric spaces is at the same time a homomorphism between the echeloned spaces induced by them, cf. Example~\ref{ex_1LipNOTSimHom}. 

Let $\Lip^{\leq}_1(\calM, \calN)$ denote the set of all 1-Lipschitz maps between the metric spaces $\calM$ and $\calN$ that preserve the echelon relations of $\calM$ and $\calN$.

\begin{PROP}
Let $\bM = (M, \leq_\bM)$ and $\bN = (N, \leq_\bN)$ be two metrizable echeloned spaces. Then, there exist metric spaces $\calM$ and $\calN$, that induce $\bM$ and $\bN$, respectively, and for which
\[\Hom(\bM, \bN) = \Lip^{\leq}_1(\calM, \calN).\]

\end{PROP}

\begin{proof} The proof of Proposition~\ref{MScharacter} provides us with the existence of dull metrics $\dM$ and $\dN$ on $M$ and $N$ which induce the echelons $\leq_\bM$ and $\leq_\bN$. By rescaling $\dM$ we may assume that $\im \dM \subseteq \{0\} \cup (2, 4)$ and that $\im \dN \subseteq \{0\} \cup (1,2)$. Now, take any $f \in \Hom(\bM, \bN)$. As for any two distinct $x, y \in M$:
\[\dN (f(x), f(y)) < 2 < \dM (x,y),\]
$f$ is trivially a $1$-Lipschitz map between the metric spaces $\calM$ and $\calN$. 
Therefore $f\in\Lip^{\leq}_1(\calM, \calN)$. 
The reverse inclusion holds trivially.
\end{proof}

Already for finite echeloned spaces, there are $1$-Lipschitz maps that do not preserve the echelons, as can be seen from the example below.

\begin{EX} \label{ex_1LipNOTSimHom}
Consider the metric spaces $\calM$ and $\calN$, given by $M = \{x_1, x_2, x_3\}$ and $N = \{y_1, y_2, y_3\}$, 
and the metrics $d_M$ and $d_N$ such that 
\begin{align*}
    d_M(x_1, x_2) &= 2 \quad \text{and} \quad d_M (x_1, x_3) = d_M (x_2, x_3) = 4\\
    d_N(y_1, y_3) &= d_N (y_2, y_3) = 1 \quad \text{and} \quad d_N(y_1, y_2) = 2.
\end{align*}

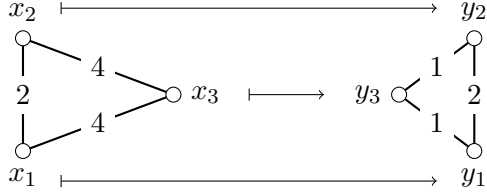
\begin{center}
\begin{tikzpicture}[scale=1,baseline=0pt]
 \GraphInit[vstyle=Classic]
		\tikzset{VertexStyle/.style = {shape = circle, draw,minimum size = 2pt, inner sep = 2pt}}        
		\tikzset{EdgeStyle/.style = {-}}
		\tikzset{>=stealth'}		
   		\Vertex[x=0,y=-0.75,Lpos=-90,L=$y_1$]{y}
   		\Vertex[x=-1,y=0,Lpos=180,L=$y_3$]{z}
   		\Vertex[x=0,y=0.75,Lpos=90,L=$y_2$]{x}
        \Edge[label=$2$](x)(y)
        \Edge[label=$1$](y)(z)
        \Edge[label=$1$](x)(z)
        \Vertex[x=-6,y=-0.75,Lpos=-90,L=$x_1$]{y1}
   		\Vertex[x=-4,y=0,L=$x_3$]{z1}
   		\Vertex[x=-6,y=0.75,Lpos=90,L=$x_2$]{x1}
        \Edge[label=$2$](x1)(y1)
        \Edge[label=$4$](y1)(z1)
        \Edge[label=$4$](x1)(z1) 
        \draw [|-to](-3,0) -- (-2,0);
        \draw [|-to](-5.5,1.15) -- (-0.5,1.15);
        \draw [|-to](-5.5,-1.15) -- (-0.5,-1.15);
         \end{tikzpicture}
         \end{center}

Any map  $f\colon \calM\to\calN$ is $1$-Lipschitz, but there are maps that do not preserve echelons, for instance $f\colon x_i\mapsto y_i$.  
\end{EX}

\begin{PROP}\label{prop_hatf_id}
Let $\bX$ be an echeloned space with $E(\bX)$ well-ordered. If $f$ is an automorphism of $\bX$, then $\hat{f}\colon E(\bX)\to E(\bX)$ is the identity embedding.
\end{PROP}
\begin{proof}
Let $f$ be an automorphism of $\bX$. As a consequence of Corollary~\ref{cor_sim_EMB}, $\hat{f}$ is an automorphism of the echeloning $E(\bX)$. Since $E(\bX)$ is well-ordered, a standard induction argument shows that $\hat{f}=\id_{E(\bX)}$.
\end{proof}

\begin{COR}\label{cor_aut=iso1}
Let $(X, d_X)$ be a metric space with $(d_X[X^2], \leq)$ well-ordered. Then the automorphism group of the induced echeloned space is the same as its isometry group, \ie $\Aut(X, \leq_\bX) = \Iso(X, d_X)$. 
\end{COR}
\begin{proof}
This follows immediately from Proposition~\ref{prop_hatf_id} and Corollary~\ref{cor_sim_EMB}. 
\end{proof}

\begin{COR}\label{cor_aut=iso2}
Let $\bX$ be a finite echeloned space induced by a metric space $(X, d_X)$. Then $\Aut(\bX) = \Iso(X, d_X)$. 
\end{COR}
\begin{proof}
Since every finite linear order is a well-order, this follows immediately from Corollary~\ref{cor_aut=iso1}.
\end{proof}

Despite Corollaries~\ref{cor_aut=iso1} and \ref{cor_aut=iso2}, we shall see in Section~\ref{sec_Fraisse}, that the \Fraisse  limit of the class of all finite echeloned spaces is not induced by the \Fraisse limit of the class of all finite rational metric spaces, namely the rational Urysohn metric space.

\section{Echeloned structure of metric spaces}\label{metrizableEchelonedSpaces}

In this section we take a little detour by studying the echeloned structure of some nicely behaved metric spaces. We start by showing that under some mild restriction homomorphisms are uniformly continuous. 

\begin{PROP}\label{lem_f_unicont}
Let $(X, d_X)$ and $(Y, d_Y)$ be metric spaces inducing echelons $\leq_\bX$ and $\leq_\bY$, respectively. Let $f \colon (X, \leq_\bX) \to (Y, \leq_\bY)$ be such a homomorphism for which the metric space $(f[X], d_Y \restr_{f[X]})$ is not uniformly discrete. Then, $f$ is uniformly continuous.
\end{PROP}

\begin{proof}
Let $\varepsilon > 0$. Without loss of generality, assume $|X| > 1$. As $(f[X], d_Y \restr_{f[X]})$ is not uniformly discrete there exist $x_0, y_0 \in X$ such that $0 \not= d_Y(f(x_0), f(y_0)) \leq \varepsilon$. Note that $\delta \coloneqq  d_X(x_0, y_0) > 0$. We will show that for any $x, y \in X$ for which $d_X(x, y) < \delta$ it follows that $d_Y(f(x), f(y)) < \varepsilon$. 

Thus, take any $x, y \in X$, such that $d_X(x, y) < d_X(x_0, y_0)$. The latter is equivalent to $(x, y) <_\bX (x_0, y_0)$. As $f$ is a homomorphism between the induced echeloned spaces, $(f(x), f(y)) <_\bY (f(x_0), f(y_0))$, which in turn is equivalent to \[d_Y((f(x), f(y))) < d_Y(f(x_0), f(y_0)) \leq \varepsilon.\qedhere\] 
\end{proof}

\begin{COR}
Let $(X, d_X)$ be a metric space inducing the echelon $\leq_\bX$. Then any automorphism of $(X, \leq_\bX)$ is uniformly continuous.
\end{COR}

\begin{proof} 
If $(X, d_X)$ is not uniformly discrete, we obtain the claim by Proposition~\ref{lem_f_unicont}. Otherwise, the statement is trivial.
\end{proof}

The corollary above implies that the isometry group of a metric space $(X, d_X)$ is a subgroup of the automorphism group of the echeloned space $(X, \leq_\bX)$, which in turn is a subgroup of the automorphism group of the uniformity space induced by $(X, d_X)$.

The example below shows that homomorphisms between echeloned spaces need not to be uniformly continuous, even if the echeloned spaces in question are induced by metric spaces.
\begin{EX}
Let $X \coloneqq  \{0\} \cup \{\frac{1}{n} \mid n \in \NN^+\}\subseteq \mathbb{R}$ be endowed with the usual metric. Let $Y \coloneqq  \{\delta_0\} \cup \bigl\{\left(1 + \frac{1}{n}\right) \delta_0 + \delta_n \mid n \in \NN^+\bigr\}\subseteq {\ell}^{1}$ be endowed with the ${\ell}^{1}$ metric  (where $\delta_i$ denotes the sequence with all entries $0$, except for the $i$\textsuperscript{th} which is $1$). Then the map $f \colon X \to Y$ defined so as to map $0$ to $\delta_0$ and $\frac{1}{n}$ to $\left(1 + \frac{1}{n}\right) \delta_0 + \delta_n$ is surjective and also preserves the echelon induced by the metric. However, it is not continuous, let alone uniformly continuous.
\end{EX}

Next, we consider a class of metric spaces for which the automorphisms of the induced echelon are exactly the dilations, see Proposition~\ref{similarity}.

\begin{DEF}
Let $(X, d_X)$ be a metric space. A point $z \in X$ is called a \emph{midpoint} of $x$ and $y$, where $x, y \in X$, if 
\[d_X(z, x) = d_X(z, y) = \frac{1}{2}d_X(x, y).\]
The set of all midpoints of points $x, y \in X$ is denoted by $\Mid_X(x, y)$. 
The metric space itself is said to \emph{have midpoints} if for any $x, y \in X$ there exists a $z \in X$ which is their midpoint. 
\end{DEF}

\begin{LEM}\label{lem_fconst_inj}
Let $(X, d_X)$ and $(Y, d_Y)$ be metric spaces inducing echelons $\leq_\bX$ and $\leq_\bY$, respectively.  Let $f \colon (X, \leq_\bX) \to (Y, \leq_\bY)$ be a homomorphism. If $(X, d_X)$ has midpoints, then $f$ is constant or injective.
\end{LEM}

\begin{proof}
Assume $f$ is not injective. That means that there exist distinct $x_0, y_0 \in X$ such that $f(x_0) = f(y_0)$. Let $\varepsilon \coloneqq  d_X(x_0,y_0)$. Take any $x, y \in X$ and let $x=m_0,m_1, m_2, \dots, m_{2^k}=y$ be points from $X$ such that $d_X(m_i, m_{i+1}) = \frac{d_X(x, y)}{2^k} \leq \varepsilon=d_X (x_0,y_0)$, for any $i \in \{0, 1, \dots, 2^k-1\}$. The latter points exist due to $(X, d_X)$ having midpoints. Consequently, for all $i \in \{0, 1, \dots, 2^k-1\}$ we have $(m_i, m_{i+1}) \leq_\bX (x_0, y_0)$. Further, since $f$ is a homomorphism, it follows that $(f(m_i), f(m_{i+1})) \leq_\bY (f(x_0), f(y_0)) = (f(x_0), f(x_0))$. Thus, $f(m_i) = f(m_{i+1})$ for all $i \in \{0, 1, \dots, 2^k-1\}$, leading to $f(x) = f(y)$. From the arbitrary choice of $x$ and $y$, we conclude that $f$ is constant.
\end{proof}

\begin{LEM}\label{lem_midpoints}
Let $(X, d_X)$ and $(Y, d_Y)$ be metric spaces inducing echelons $\leq_\bX$ and $\leq_\bY$, respectively.  Let $f \colon (X, \leq_\bX) \to (Y, \leq_\bY)$ be a surjective homomorphism. If $(X, d_X)$ and $(Y, d_Y)$ have midpoints, then for all $x, y \in X$
\[f\left[\Mid_X(x, y)\right] \subseteq \Mid_Y(f(x), f(y)).\]
\end{LEM}
\begin{proof}
For constant $f$ the lemma trivially holds. So, by Lemma~\ref{lem_fconst_inj}, it remains to handle the case that $f$ is injective.  Take any $x, y \in X$ and denote by $m$ a midpoint of $x$ and $y$. Then we have that $d_X(x, m) = d_X(m, y) = \frac{1}{2} d_X(x, y)$. Thus, $(x, m) \sim_\bX (m, y)$, and so $(f(x), f(m)) \sim_\bY (f(m), f(y))$, \ie $d_Y(f(x), f(m)) = d_Y(f(m), f(y))$. Now, from the triangle inequality we know that $d_Y(f(x), f(y)) \leq d_Y(f(x), f(m)) + d_Y(f(m), f(y)) = 2 d_Y(f(x), f(m))$. Hence, $\frac{1}{2} d_Y(f(x), f(y)) \leq d_Y(f(x), f(m))$. 

In what follows, we shall show that $f(m)$ is in fact a midpoint of $f(x)$ and $f(y)$.
Assume the opposite, that $d_Y(f(x), f(m)) > \frac{1}{2} d_Y(f(x), f(y))$. As $f$ is surjective and $(Y, d_Y)$ has midpoints too, there must exist some $m' \in X$ such that $f(m')$ is a midpoint of $f(x)$ and $f(y)$: 

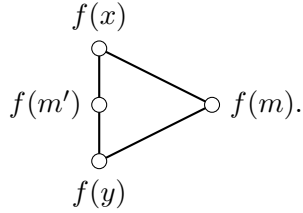
\begin{center}
\begin{tikzpicture}[scale=1,baseline=0pt]
 \GraphInit[vstyle=Classic]
		\tikzset{VertexStyle/.style = {shape = circle, draw,minimum size = 2pt, inner sep = 2pt}}        
		\tikzset{EdgeStyle/.style = {-}}
		\tikzset{>=stealth'}		
   		\Vertex[x=0,y=-0.75,Lpos=-90,L=$f(y)$]{y}
   		\Vertex[x=1.5,y=0,L=$f(m)$.]{z}
   		\Vertex[x=0,y=0.75,Lpos=90,L=$f(x)$]{x}
        \Vertex[x=0,y=0,Lpos=180,L=$f(m')$]{u}
        \Edges(x,u,y,z,x)
         \end{tikzpicture}
         \end{center}
As $f(m) \not= f(m')$, it follows that $m \not= m'$. 
Suppose $d_X(x, m') \not= d_X(m', y)$. Without loss of generality, let $d_X(x, m') < d_X(m', y)$. 
Note that by the triangle inequality it is impossible that $d_X(x,m)\geq d_X(m',y)$. Therefore, there are two possible cases:
\begin{enumerate}[label=\textbf{Case \arabic*:}, ref=Case \arabic*,nosep,align=left,leftmargin=0em,labelindent=0em,itemindent=1em,labelsep=0.5em,labelwidth=!]
    \item $d_X(x, m) \leq d_X(x, m') < d_X(m', y)$ --- Then,  
$(f(x), f(m)) \leq_\bY (f(x), f(m')) \leq_\bY (f(m'), f(y))$. However, this leads to \[d_Y(f(x), f(m)) \leq d_Y(f(x), f(m')) = \frac{1}{2} d_Y(f(x), f(y)),\] which in turn contradicts our assumption $f(m)\not\in \Mid_Y(f(x), f(y))$;
    \item $d_X(x, m') \leq d_X(x, m) < d_X(m', y)$ --- Then  
    $(f(x), f(m')) \leq_\bY (f(x), f(m)) \leq_\bY (f(m'), f(y))$. In other words, $d_Y(f(x), f(m')) \leq d_Y(f(x), f(m)) \leq d_Y(f(m'), f(y))$. Nonetheless, this implies that $d_Y(f(x), f(m)) = d_Y(f(x), f(m')) = \frac{1}{2}d_Y(f(x), f(y))$ which is again a contradiction to our assumption that $f(m)\not\in \Mid_Y(f(x), f(y))$.
\end{enumerate}
All in all, we have just shown that $d_X(x, m') = d_X(m', y)$, \ie $(x, m') \sim_\bX (m', y)$. Therefore, $d_X(x, y) \leq d_X(x, m') + d_X(m', y) = 2 d_X(x, m')$, and so $d_X(x, m) = \frac{1}{2} d_X(x, y) \leq d_X(x, m')$, \ie $(x, m) \leq_\bX (x, m')$.  As a result $(f(x), f(m)) \leq_\bY (f(x), f(m'))$, \ie $d_Y(f(x), f(m)) \leq d_Y(f(x), f(m')) =\frac{1}{2} d_Y(f(x), f(y))$, thus we reach a contradiction. 

So our initial assumption was wrong and $f(m)$ is indeed a midpoint of $f(x)$ and $f(y)$, as desired.
\end{proof}

\begin{REM}\label{rem_collinear}
    Let $(X,d_X)$ be a metric space. Recall that a triple $(a,b,c)\in X^3$ is said to be \emph{collinear} if 
    \[
    d_X(a,b)+d_X(b,c) = d_X(a,c).
    \]
    Furthermore it is well-known and easy to see that for any four points $a,b,c,d\in X$ joint collinearity of $(a,b,d)$ and $(b,c,d)$ implies the collinearity of $(a,b,c)$ and of $(a,c,d)$ (see \cite[Section 2]{Men28}, cf.\ \cite[Section 6]{Chv04}). 
\end{REM}

\begin{PROP}\label{similarity}
Let $f$ be an automorphism of the echeloned space $(X, \leq_\bX)$, induced by a metric space $(X, d_X)$ that has midpoints. Then $f$ is a dilation, \ie there exists a positive real number $t$ such that for all $x, y \in X$ we have that $d_X(f(x), f(y)) = t \cdot d_X(x, y)$. 
\end{PROP}
\begin{proof}
  Let $\hat{f}$ be the action of $f$ on the echeloning of $(X, \leq_\bX)$ (cf.~Corollary~\ref{cor_sim_EMB}). Define $\tilde{d}_X \colon E(X)\to \im d_X,\, [(x,y)]_{\sim_{\bX}}\mapsto d_X(x,y)$. Since $\leq_\bX$ is induced by $d_X$, it is easy to see that $\tilde{d}_X$ is an order isomorphism (cf.~Lemma~\ref{MSinduceS}). Let $\tilde{f}\coloneqq \tilde{d}_X\circ \hat{f}\circ \tilde{d}_X^{-1}$. In particular, for all $a,b\in X$ we have $\tilde{f}(d_X (a,b))=d_X (f(a),f(b))$, and the following diagram commutes: 
\[\begin{tikzcd}
X^2 \arrow[r, twoheadrightarrow, "\eta_\bX"]\arrow[d, rightarrow, hook, "f^2"{name=L, left}] \arrow[rr, rightarrow, bend  left, "d_X"] & E(\bX)\arrow[d, rightarrow, hook, "\hat{f}"] \arrow[r, rightarrow, dashed,"\tilde{d}_X","\cong"'] & \im d_X\arrow[d, rightarrow, hook, "\tilde{f}"]\\
X^2 \arrow[r, twoheadrightarrow, "\eta_\bX"]\arrow[rr, rightarrow, bend  right, "d_X"]& E(\bX)\arrow[r, rightarrow, dashed,"\tilde{d}_X","\cong"'] & \im d_X.
\end{tikzcd}
\]

We will show that for every $\delta \in \im d_X\setminus \{0\}$ the restriction of $\tilde{f}$ to the initial segment $I(\delta) \coloneqq  [0, \delta] \cap \im d_X$ is linear. Since, for any $0 < \delta < \delta' \in \im d_X$, we have $\{0\} \neq I(\delta) \subseteq I(\delta')$, it will follow that $\tilde{f}$ as a whole is linear, \ie $f$ is a dilation.

Let $t \coloneqq  \tilde{f}(\delta) / \delta$. Let $a, c \in X$ with $ d_X(a, c) = \delta$, and therefore with $t = \tfrac{ d_X (f(a), f(c))}{ d_X (a, c)}$. Let $b$ be a midpoint of $a$ and $c$. 

Let $s\in I(\delta)$. Since $\tilde{f}(0)=0=t\cdot 0$, we may and will assume that $s\neq 0$. We choose recursively a sequence $(b_n,c_n)_{n\in\NN}$ in $X^2$ with
\[
\forall n\in\NN\,:\, s\le d_X(a,c_n), \text{ and } (a,b_n,c_n) \text{ is collinear.}
\]
We proceed as follows: $(b_0,c_0)\coloneqq (b,c)$. If $(b_n,c_n)$ has been chosen, then $(b_{n+1},c_{n+1})$ is chosen according to the following cases:
\begin{enumerate}[label=\textbf{Case \arabic*:}, ref=\arabic*,nosep,align=left,leftmargin=0em,labelindent=0em,itemindent=1em,labelsep=0.5em,labelwidth=!]
    \item if $s=d_X(a,c_n)$, then put $(b_{n+1},c_{n+1})\coloneqq (c_n,c_n)$,
    \item if $d_X(a,b_n)<s<d_X(a,c_n)$ then choose $m\in\Mid(b_n,c_n)$ and put
    \[
    (b_{n+1},c_{n+1})\coloneqq \begin{cases}
        (m,m) & \text{if } s=d_X(a,m),\\
        (b_n,m) & \text{if } s<d_X(a,m),\\
        (m,c_n) & \text{if } s>d_X(a,m),
    \end{cases}
    \]
    \item if $d_X(a,b_n)=s$ then put $(b_{n+1},c_{n+1})\coloneqq (b_n,b_n)$,
    \item if $0< s < d_X(a,b_n)$ then choose $m\in\Mid(a,b_n)$ and put $(b_{n+1},c_{n+1})\coloneqq (m,b_n)$.
\end{enumerate}
Let us call a pair $(x,y)\in X^2$ \emph{homothetic} if 
\begin{enumerate}[label=\arabic*), ref=\arabic*)]
    \item $(a,x,y)$ is collinear,
    \item $(f(a),f(x),f(y))$ is collinear,
    \item $d_X(f(a),f(x))=t\cdot d_X(a,x),\quad d_X(f(x),f(y)) = t\cdot d_X(x,y)$.
\end{enumerate}
A straight forward induction using Remark~\ref{rem_collinear} and Lemma~\ref{lem_midpoints} shows that for every $n\in\NN$ the pair $(b_n,c_n)$ is homothetic. Furthermore
\[
s = \lim_{n\to\infty} d_X(a,b_n) = \lim_{n\to\infty} d_X(a,c_n).
\]
Note that since $s>0$, for all but finitely many $n$ we have $d_X(a,b_n)\le s$. 

Since $\tilde{f}(d_X(a,b_n)) = t\cdot d_X(a,b_n)$ and $\tilde{f}(d_X(a,c_n)) = t\cdot d_X(a,c_n)$ for all $n\in\NN$, monotonicity of $\tilde{f}$ implies that $\tilde{f}(s)=t\cdot s$, as desired. 
\end{proof}

What follows is an example of a metric space  and an automorphism of the induced echeloned space which is not a dilation.

\begin{EX}
We define a subspace $X$ of the Euclidean line $\RR$ as follows. Let $a > 2$ and $0 < \epsilon < \tfrac{a}{2} - 1$. Let $x_0 = 1 + \epsilon$ and, for any $k \in \ZZ \setminus \{0\}$, $x_k = a^k$. Let finally $X = \{x_k\mid k \in\ZZ\}$ be endowed with the usual Euclidean metric. Observe that $X$ does not have midpoints since it is discrete.

We claim that the shift $\sigma\colon X \rightarrow X,\, x_k \mapsto x_{k + 1}$ is an automorphism of the echelon relation $\leq_\bX$ of $X$. Obviously, $\sigma$ is not a dilation, since $\epsilon \neq 0$.

To check that it is indeed an automorphism, let us consider two distinct pairs of distinct points $\{x_k, x_\ell\}$ and $\{x_{k'}, x_{\ell'}\}$. We may assume without loss of generality that $k > \ell$ and $k' > \ell'$, and moreover $k \geq k'$. Two cases are to be considered:
\begin{enumerate}[label=(\roman*), ref=(\roman*)]
    \item If $k = k'$, then $(x_{k'}, x_{\ell'}) \leq_\bX (x_k, x_\ell)$ if and only if $\ell \leq \ell'$. This is obviously equivalent to $k + 1 = k' + 1$ and $\ell + 1 \leq \ell' + 1$, \ie to $(\sigma(x_{k'}), \sigma(x_{\ell'})) \leq_\bX (\sigma(x_k), \sigma(x_\ell))$.
    \item If $k > k'$, then necessarily $(x_{k'}, x_{\ell'}) < (x_k, x_\ell)$. Indeed, we have
\begin{align*}
     d_X (x_k, x_\ell)  &\geq  d_X (x_k, x_{k - 1}) = \begin{cases} a^{k - 1} (a - 1) & \text{if } k \neq 1, \\ a - 1 - \epsilon & \text{if } k = 1, \end{cases} \\ 
     d_X (x_{k'}, x_{\ell'})  &\leq x_{k'} \leq x_{k - 1} = \begin{cases} a^{k - 1} & \text{if } k \neq 1, \\ 1 + \epsilon & \text{if } k = 1. \end{cases} 
\end{align*}
And therefore $ d_X (x_k, x_\ell) >  d_X (x_{k'}, x_{\ell'})$ whatever $k$, thanks to our choice of $\epsilon$ and $a$. Since obviously $k + 1 > k' + 1$, the same argument would hold for the images by the shift, \ie $(\sigma(x_{k'}), \sigma(x_{\ell'})) <_\bX (\sigma(x_k), \sigma(x_\ell))$.
\end{enumerate}
\end{EX}

\section{The \Fraisse~limit of the class of finite echeloned spaces}\label{sec_Fraisse}

We now turn our attention to the proof of the existence of a countable universal homogeneous echeloned space.

\begin{PROP}\label{prop_SIMhasAP}
The class of finite echeloned spaces has the amalgamation property.
\end{PROP}
\begin{proof}
Consider any three finite echeloned spaces $\bA, \bB_1, \bB_2$ together with embeddings $f_1 \colon \bA \embedsto \bB_1$ and $f_2 \colon \bA \embedsto \bB_2$. 
By Corollary~\ref{cor_sim_EMB}, we obtain embeddings $\hat{f}_1 \colon E(\bA) \embedsto E(\bB_1)$ and $\hat{f}_2 \colon E(\bA) \embedsto E(\bB_2)$, for which the following equations hold: \[\hat{f_1} \circ \eta_\bA = \eta_{\bB_1} \circ {f_1}^2 \quad \text{and} \quad \hat{f_2} \circ \eta_\bA = \eta_{\bB_2} \circ {f_2}^2.\]
Notice that $E(\bA)$, $E(\bB_1)$ and $E(\bB_2)$ are all (finite) linear orders.  Given that the class of finite linear orders has the amalgamation property, there exist a linear order $(D, \sqsubseteq)$ and embeddings 
\[{g}_1 \colon E(\bB_1) \embedsto (D, \sqsubseteq) \quad \text{and} \quad {g}_2 \colon E(\bB_2) \embedsto (D, \sqsubseteq)\]
for which ${g}_1 \circ \hat{f}_1 = {g}_2 \circ \hat{f}_2$, \ie for which this diagram commutes:
\[\begin{tikzcd}
& E(\bB_1) \arrow[dr, dashed, hook, "{g}_1"] & \\
E(\bA)\arrow[ur, hook, "\hat{f}_1"] \arrow[dr, hook, "\hat{f}_2"']& &(D, \sqsubseteq). \\
& E(\bB_2)\arrow[ur, dashed, hook, "{g}_2"'] &
\end{tikzcd}
\]
For each $i\in\{1,2\}$ observe that $\min  E(\bB_i)=\hat{f}_i (\min E(\bA))$ and therefore, without loss of generality, we can assign  $g_i(\min  E(\bB_i))=\min (D, \sqsubseteq)$. Moreover, we may assume that $\min (D, \sqsubseteq)\neq \max (D, \sqsubseteq)$. Now, define $C \coloneqq  B_1 \dotcup (B_2 \setminus f_2[A])$ and $\eta \colon C^2 \to (D, \sqsubseteq)$ as:
\[
\eta(x, y) \coloneqq  \begin{cases}
            {g}_1 (\eta_{\bB_1} (x, y)) & \text{if } (x, y) \in B^2_1, \\
            {g}_2 (\eta_{\bB_2} (x, y)) & \text{if } (x, y) \in B^2_2, \\
            \max (D, \sqsubseteq) & \text{otherwise}.
           \end{cases}
\]
Observe that $\eta$ is well-defined. Now, we define a binary relation $\leq_\bC$ on $C^2$ as follows:
\[(c_1, c_2) \leq_{\bC} (c_1', c_2') \quad :\Longleftrightarrow \quad \eta(c_1, c_2) \sqsubseteq \eta(c_1', c_2').\]

What we need to show now is that $\bC \coloneqq  (C, \leq_\bC)$ is indeed an echeloned space.  The only nontrivial point to show is axiom (ii) from Definition~\ref{orgDEF_sim}. Take any $c_0, c_1, c_2 \in C$. 
Let us assume that  $(c_1, c_2) \leq_{\bC} (c_0, c_0)$, then $\eta(c_1, c_2) \sqsubseteq \eta(c_0, c_0) = \min (D, \sqsubseteq)$, and so $\eta(c_1, c_2) = \min (D, \sqsubseteq)$. In other words, $(c_1, c_2) \in \bB_i^2$ for some $i \in \{1, 2\}$. 
Therefore axiom (ii) for $\bB_1$ or $\bB_2$ implies $c_1=c_2$.
Thus, $\bC$ is a well-defined finite echeloned space.

It remains to show that the dashed arrows in the following commuting diagram are embeddings:
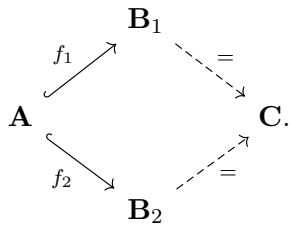
\[\begin{tikzcd}
&\bB_1 \arrow[dr, dashed, "="]& \\
\bA\arrow[ur, hook, "f_1"] \arrow[dr, hook, "f_2"']& & \bC.\\ 
&\bB_2\arrow[ur, dashed, "="']&
\end{tikzcd}
\]
Indeed, for any choice of $b_1, b_2, b_1', b_2' \in \bB_i$, for $i \in \{1, 2\}$, $(b_1, b_2) \leq_{\bB_i} (b_1', b_2')$ is equivalent to $\eta_{\bB_i} (b_1, b_2)  \leq_{E(\bB_i)} \eta_{\bB_i} (b_1', b_2')$ which in turn is equivalent to $\eta(b_1, b_2)  \sqsubseteq \eta(b_1', b_2')$, \ie to $(b_1, b_2) \leq_\bC (b_1', b_2')$. This concludes the proof.
\end{proof}

\begin{REM}\label{echelonhasSAP}
    The proof of Proposition~\ref{prop_SIMhasAP} actually shows that the class of finite echeloned spaces has the strong amalgamation property.
\end{REM}

\begin{PROP}\label{thm_SIMisFraisse}
The class of finite echeloned spaces is a \Fraisse~class.
\end{PROP}

\begin{proof}
What we need to show is that the class of finite echeloned spaces has the hereditary property ($\HP$), the joint embedding property ($\JEP$), the amalgamation property ($\AP$), and that up to isomorphism there exist just countably many finite echeloned spaces. 

The $\AP$ was already established in Proposition~\ref{prop_SIMhasAP}. With regards to establishing the $\JEP$: observe that, by definition, an echeloned space is defined on a non-empty set. Note that the trivial one element echeloned space is embeddable into any echeloned space. Therefore, in this case, the $\JEP$ follows from the $\AP$.
Any subset of an echeloned space induces an echeloned space. Thus, this class has the $\HP$. Having been defined over a finite relational signature, it clearly has only countably many isomorphism classes. 
\end{proof}

Proposition~\ref{thm_SIMisFraisse}, together with \Fraisse's theorem, asserts the existence of a unique (up to isomorphism) countable universal homogeneous echeloned space. We will denote this \Fraisse limit by $\bF = (F, \leq_\bF)$ and the smallest element of its echeloning by $\bot_\bF$.

We proceed by examining the structure of $\bF$ in more detail. 

\begin{LEM}\label{distQ}
$E(\bF) \cong (\QQ^+_0, \leq)$.
\end{LEM}

\begin{proof}
What we will show is that excluding $\bot_\bF$ from the echeloning of $\bF$ leaves us with a countable dense linear order without endpoints. In other words, the original structure is isomorphic to the set of non-negative rational numbers equipped with the natural order.

As $\bF$ is an echeloned space,  $E(\bF)$ is a linear order. 
Take any $u_1, u_2, u_3, u_4 \in F$ such that $(u_1, u_2) <_\bF (u_3, u_4)$. Define $U \coloneqq  \{u_1, u_2, u_3, u_4\}$ and let $\bU$ be the echeloned subspace of $\bF$ induced by $U$. Further, let $v \not\in F$ be a new point and $V \coloneqq  U \cup \{v\}$. We define an echeloned superspace of $\bU$ on $V$, and call it $\bV$, for which  $(u_1, u_2) <_\bV (v, u_i)\sim_\bV (v,u_j) <_\bV (u_3, u_4)$, for any $i,j \in \{1, 2, 3, 4\}$. 

Since $\bF$ is homogeneous, it is also weakly homogeneous in the sense of \cite[p.~326]{hodges1993model}. Consequently, as $\bF$ is also universal, there exists $\iota\colon \bV\to \bF$, such that the following diagram commutes:
\[\begin{tikzcd}
 \bU \arrow[r, hook, "="] \arrow[dr, hook, " ="']& \bV\arrow[d, hook, dashed, "\iota"]\\
  &\bF.
\end{tikzcd}
\]

Let $u_5 \coloneqq  \iota(v)$. Then, $(u_1, u_2) <_\bF (u_1, u_5) <_\bF (u_3, u_4)$. Thus, $E(\bF)$ is indeed dense.

Similarly, one can prove the non-existence of both the smallest and the greatest element of $E(\bF) \setminus \{\bot_\bF\}$.
\end{proof}

\begin{PROP}\label{prop_leqS_buMS}
Every metric space that induces $\leq_\bF$ is  dull (and consequently bounded and uniformly discrete).
\end{PROP}

\begin{proof} Let $(F,  d)$ be some metric space which induces $\bF$.  
Further, let $\phi \colon \im d \to E(\bF)$ be the naturally induced order-isomorphism. Define $\tau \coloneqq  \inf\{ d(x, y) \mid x, y \in F, x \not= y\}$. Recall that being dull is equivalent to $\im d\subseteq [\tau,2\tau]\cup \{0\}$ (see the proof of Lemma~\ref{lem_dullMS}).

Fix distinct $x_0, y_0 \in F$. Recall that $\bF$ is universal, thus for any choice of distinct $x, y \in F$ there have to exist $u,  v, w \in F$ such that 
\[
(u, v) \sim_\bF (x, y), \quad  (u, w) \sim_\bF (v, w) \sim_\bF (x_0, y_0).\] 
Put differently, $ d(u, v) = c$ and $d(u, w) =  d(v, w) = c_0$, where $c_0 \coloneqq   d(x_0, y_0)$ and $c \coloneqq   d(x, y)$:
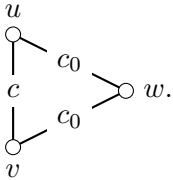
\begin{center}
\begin{tikzpicture}[scale=1,baseline=0pt]
 \GraphInit[vstyle=Classic]
		\tikzset{VertexStyle/.style = {shape = circle, draw,minimum size = 2pt, inner sep = 2pt}}        
		\tikzset{EdgeStyle/.style = {-}}
		\tikzset{>=stealth'}		
   		\Vertex[x=0,y=-0.75,Lpos=-90,L=$v$]{y}
   		\Vertex[x=1.5,y=0,L=$w$.]{z}
   		\Vertex[x=0,y=0.75,Lpos=90,L=$u$]{x}
        \Edge[label=$c_0$](y)(z)
        \Edge[label=$c_0$](z)(x)
        \Edge[label=$c$](x)(y)
         \end{tikzpicture}
         \end{center}

By the triangle inequality, we get $c\leq 2c_0$.

Since $(x,y)$ (and hence $c$) was arbitrary, we have $\im d\subseteq [0,2c_0]$.

By letting $c_0$ converge to $\tau$ we obtain that $\im d\subseteq [0,2\tau]$. By the definition of $\tau$ we actually have $\im d\subseteq \{0\}\cup [\tau,2\tau]$. Hence, $(F,  d)$ is dull.
\end{proof}

\begin{COR}\label{notUry}
$\bF$ is not isomorphic to the echeloned space induced by the (bounded) rational Urysohn space.
\end{COR}

\begin{proof}
As the (bounded) rational Urysohn space is not dull, it follows by Proposition~\ref{prop_leqS_buMS} that it cannot induce $\bF$.
\end{proof}

However, the echeloned space $\bF$ is indeed induced by a \emph{dull Urysohn space}, as we now explain. 
%This does not preclude interesting automorphisms to occur (see Section~\ref{sec_Ramsey} and in particular Proposition~REF) since a dull metric space does not have midpoints.

Let $S = \{0\} \cup (1, 2) \cap \QQ$. Consider the $S$-Urysohn space $\bU_S = (U,  d_{\bU_S})$, \ie the \Fraisse limit of all finite metric spaces with distances in $S$ (see~\cite[Theorem 1.4]{Sauer}). Let $\bM =  (U, \leq_\bM)$ be the echeloned space induced on $U$ by $ d_{\bU_S}$.

\begin{PROP}
    The echeloned space $\bM$ is isomorphic to $\bF$.
\end{PROP}

\begin{proof}
By the uniqueness of \Fraisse limits, it is enough to prove that $\bM$ is weakly homogeneous and universal for the class of all finite echeloned spaces.

\begin{description}
    \item[Universality] Obviously, every finite substructure of $\bM$ is a finite echeloned space. Conversely, let $\bA = (A, \leq_\bA)$ be a finite echeloned space. By Proposition~\ref{MScharacter}, we can realize $\leq_\bA$ by a (dull) metric $ d_\bA$, whose image is in $\{0\} \cup (1, 2)$. Since this image is finite, we can moreover assume that $ d_\bA$ only takes rational values (up to applying an order-automorphism of $\{0\} \cup (1, 2)$). We have therefore realized $\bA$ as a metric space with distances in $S$. By the universality of the Urysohn space $\bU_S$, there exists an isometric embedding $f$ of $(A,  d_\bA)$ into $(U,  d_{\bU_S})$. Such an embedding is also an embedding of the respective induced echeloned spaces, \ie $\bA$ indeed embeds into $\bM$.
    \item[Weak homogeneity] Let $\bA =(A, \leq_\bA)$ and $\bB = (B, \leq_\bB)$ be two finite echeloned spaces such that $\bA\le\bB$. Let $f$ be an embedding of $\bA$ into $\bM$. We can realize $\bA$ as a metric space by using the metric inducing $\leq_\bM$, \ie by defining
\begin{equation*}
     d_\bA (x, y) \coloneqq   d_{\bU_S} (f(x), f(y)) \qquad (x, y \in A).
\end{equation*}
This trivially makes $f$ an isometric embedding of $(A,  d_\bA)$ into $(U,  d_{\bU_S})$. Now choose any metric $ d_\bB$ on $B$, extending $ d_\bA$, with image in $S$ and inducing $\leq_\bB$ (such a metric exists since $\bB$ is finite and the order of $(1, 2) \cap \QQ$ is dense and without minimum nor maximum). By the universality and the weak homogeneity of the Urysohn space $\bU_S$, there exists an isometric embedding $g$ of $\bB$ into $\bU_S$ such that the diagram
\begin{equation*}\begin{tikzcd}
(A,  d_\bA) \arrow[dr, hook, "f"'] \arrow[r, hook, "="]& (B,  d_\bB) \arrow[d, dashed, hook, "g"]\\
  &  (U,  d_{\bU_S}).
\end{tikzcd}
\end{equation*}
commutes. Obviously, $g$ is also an embedding of the induced echeloned space.\qedhere
\end{description}
\end{proof}

\begin{REM}
Let $\tau_1, \tau_2 > 0$ and, for $i = 1, 2$, let $T_i = T_i' \cup \{0\}$, where $T_i'$ is any countably infinite subset of $(\tau_i, 2 \tau_i)$ which is order-dense and without maximum nor minimum. As shown by the above proposition, the $T_i$-Urysohn space $\bU_{T_i}$ induces the echeloned space $\bF$, independently of $i$. It follows that the isometry groups $\Iso(\bU_{T_1})$ and $\Iso(\bU_{T_2})$ are isomorphic (since they are nothing but the subgroup of $\Aut(\bF, \leq_\bF)$ that fixes the echeloning), whereas the spaces $\bU_{T_1}$ and $\bU_{T_2}$ are not isometric as soon as $T_1 \neq T_2$.
\end{REM}

Echeloned spaces naturally give rise to edge coloured graphs where every edge is coloured by its equivalence class in the corresponding echeloning. We describe now this graph for $\bF$.

\begin{DEF}
Let $C$ be a non-empty set. A \emph{$C$-coloured graph} $\Gamma$ is an ordered pair $(V, \chi)$, where $V$ is a set  and $\chi \colon [V]^2 \to C$ (here and in the following we adopt the standard notation that $[V]^2$ refers to the set of $2$-element subsets of $V$).  $V$ is called  the \emph{set of vertices} and $\chi$ is called the \emph{edge-colouring function} of $\Gamma$.
\end{DEF}

Let us stress that $C$-coloured graphs are by definition complete graphs.

\begin{DEF}
Let $\Gamma_1 = (V_1, \chi_1)$ and $\Gamma_2 = (V_2, \chi_2)$ be two $C$-coloured graphs, where $C$ is a fixed set of colours. Then a \emph{homomorphism} from $\Gamma_1$ to $\Gamma_2$ is an injective map $f\colon V_1\to V_2$ such that
\[ \forall \{x, y\} \in [V_1]^2\,:\,
\chi_1(\{x, y\}) = \chi_2(\{f(x), f(y)\}).\]
\end{DEF}

Every $C$-coloured graph $\Gamma = (V, \chi)$ may be defined equivalently as a relational structure $\mathbf{\Gamma}$ over the signature $\{\varrho_c \mid c \in C\}$ consisting solely of binary relation symbols, where $\varrho_c^{\mathbf{\Gamma}} \coloneqq  \{(u, v) \in V^2 \colon u\neq v,\chi(\{u, v\}) = c \}$. In particular, a function $h \colon V_1 \to V_2$ is a homomorphism from $\Gamma_1$ to $\Gamma_2$ if and only if it is a homomorphism from $\mathbf{\Gamma}_1$ to $\mathbf{\Gamma}_2$. This allows us to identify every $C$-coloured graph $\Gamma$ with its associated relational structure $\mathbf{\Gamma}$. We shall freely do so without any further notice.

What follows is a characterisation of homogeneous universal countable $C$-coloured graphs, for countable $C$, following \cite{truss1985group} and \cite{tarzi2014multicoloured}.

\begin{DEF}
Let $\Gamma = (V, \chi)$ be a $C$-coloured graph, and let $k$ be a positive integer. We say that $\Gamma$ has the \emph{$(\ast_k)$-property} if for any $(c_1, c_2, \ldots, c_k) \in C^k$ and any
choice of $k$ finite, pairwise disjoint subsets of $V$, denoted $U_1, U_2, \ldots, U_k$, there exists a vertex $z \in V$ such that for all $i \in \{1, 2, \dots, k\}$ and all $u \in U_i$: $\chi(\{z, u\}) = c_i$.
$\Gamma$ is said to have the \emph{$(\ast_\infty)$-property} if it has the $(\ast_k)$-property for all $k \in \NN^+$.
\end{DEF}

Note that for every $k\in\mathbb{N}^+$ the $(\ast_{k+1})$-property implies the $(\ast_k)$-property.

It is straightforward to check that the class of all finite $C$-coloured graphs enjoys the amalgamation property, so we obtain the following:

\begin{LEM}
Let $C$ be a countable set of colours. Then the class of all finite $C$-coloured graphs is a \Fraisse~class. 
\end{LEM}

For every countable $C$ we will denote the universal homogeneous $C$-coloured graph, \ie the \Fraisse~limit of the class of all finite $C$-coloured graphs, by $\Tc = (V_C, \chi_C)$.

\begin{PROP}\label{prop_*inf}
Let $C$ be countable. Then, a countable $C$-coloured graph has the $(\ast_\infty)$-property if and only if it is homogeneous and universal for the class of finite $C$-coloured graphs.
\end{PROP}

\begin{proof} 
``$\Rightarrow$'': Let $\Gamma = (V, \chi)$ be a countable $C$-coloured graph for which the $(\ast_\infty)$-property holds. 

Let us start by showing that $\Gamma$ is universal. We proceed by induction on the size $\ell$ of the $C$-coloured graph to be embedded. The empty $C$-coloured graph embeds to $\Gamma$ trivially. Suppose that $\Delta=(W,\eta)$ is a $C$-coloured graph of size $\ell+1$. Let $v\in W$. Let $\Delta'$ be the $C$-coloured subgraph of $\Delta$ induced by $W\setminus \{v\}$, and suppose that $\Delta'$ embeds into $\Gamma$ by an embedding $\iota$. Let $c_1,\dots,c_k$ be all the (pairwise distinct) colours appearing as a colour of an edge from $v$ in $\Delta$. For each $i\in \{1,\dots,k\}$ let $U_i\coloneqq \{x\in W\setminus \{v\}\mid \eta (\{x,v\})=c_i\}$. Note that $\iota (U_1),\dots,\iota (U_k)$ are pairwise disjoint. By the $(*_{\infty})$-property (and, therefore, $(*_{k})$-property) there exists a vertex $z\in V$ such that for all $i\in \{1,\dots,k\}$ and all $u\in \iota (U_i)\,:\, \chi (\{z,u\})=c_i$. Hence, the map $\hat{\iota}\colon W\to V$ extending $\iota$ and sending $v$ to $z$ is an embedding of $\Delta$ into $\Gamma$. This finishes the proof that $\Gamma$ is universal. By iterating the argument above, it becomes clear that $\Gamma$ is also weakly homogeneous in the sense of \cite{hodges1993model}, and hence homogeneous. 

``$\Leftarrow$'': Consider any universal homogenous countable $C$-coloured graph $\Gamma = (V, \chi)$. Fix any  positive integer $k$ and a tuple $(c_1, \dots, c_k) \in C^k$. Then choose $k$ finite, pairwise disjoint subsets $U_1, \dots, U_k$ of $V$. Define a finite $C$-coloured graph on the set of vertices $V_2 \coloneqq  \bigcup\limits_{i = 1}^n U_i \cup \{u\}$, where $u \not\in V$ is a new vertex, with the edge-colouring function $\chi_2$ which maps $\{u, v\}$ to $c_i$ for any $v \in U_i$ and $\{v_1, v_2\}$ to $\chi(\{v_1, v_2\})$ for any two distinct $v_1, v_2 \in V_2 \setminus \{u\}$. Let $\Gamma_1$ be the $C$-coloured subgraph of $\Gamma$ induced by $\bigcup\limits_{i = 1}^n U_i$. Clearly, there exist identity embeddings $\iota_1 \colon \Gamma_1 \embedsto \Gamma$ and $\iota_2 \colon \Gamma_1 \embedsto (V_2, \chi_2)$. Since $\Gamma$ is homogeneous, it is weakly homogeneous in the sense of \cite{hodges1993model}. Hence, as $\Gamma$ is also universal, there exists an embedding $f \colon (V_2, \chi_2)\to\Gamma$ for which $\iota_1 = f \circ \iota_2$. As a result, we get that $\chi(\{f(u), f(\iota_2(v))\}) = \chi(\{f(u), \iota_1(v)\})=\chi(\{f(u), v\})=c_i$ for all $i \in \{1, \dots, k\}$. Thus the $(\ast_k)$-property of $\Gamma$ holds. 
\end{proof}

\begin{OBS} \label{*2=Rado}
When a $C$-coloured graph $\Gamma = (V, \chi)$, for  $|C| = 2$, has the $(\ast_2)$-property, then the graph $(V, E)$, with the set of edges $E$ defined for a fixed $c \in C$ as follows:
\[
(u,v) \in E  \quad :\Longleftrightarrow \quad \chi(\{u, v\})  = c,
\]
is isomorphic to the Rado graph.
\end{OBS}

We now establish a connection between the \Fraisse limit $\bF$ of finite echeloned spaces and the \Fraisse limit $\Tc$ of finite $C$-coloured graphs:

\begin{PROP}\label{C-coloured_G}
Fix $k \in \NN^+$ and choose pairwise distinct $c_1, \dots, c_k\in  E(\bF)\setminus\{\bot_\bF\}$. Define $C \coloneqq  \{c_1, \dots, c_k, c_\ast\}$, with $c_\ast \not\in E(\bF)$. Then $\Gamma_C \coloneqq  (F, \chi)$ is isomorphic to $\Tc$, where
\[
\chi \colon [F]^2 \to C,\, \{u, v\} \mapsto \begin{cases}
\eta_\bF(u, v)& \text{if } \eta_\bF(u, v) \in \{c_1, \dots, c_k\},\\
c_\ast& \text{otherwise}.
\end{cases}
\]
\end{PROP}

\begin{proof} By Proposition~\ref{prop_*inf}, $\Gamma_C $ is isomorphic to $\Tc$ if it has the $(\ast_\infty)$-property. As $|C| = k + 1$, it suffices to show  that $\Gamma_C$ has the $(\ast_{k+1})$-property. Consider thus the $(k+1)$--tuple $(c_1, \dots, c_k, c_\ast)$. Without loss of generality, assume $c_1 <_{E(\bF)} c_2 <_{E(\bF)} \dots <_{E(\bF)} c_k$. 

Let $U_1, \dots, U_k, U_* \subseteq F$ be finite and pairwise disjoint. Set $U' \coloneqq  U_1 \cup U_2 \cup \dots \cup U_k \cup U_*$. Further, let $U$ be a finite superset of $U'$ such that for all $i \in \{1 \dots, k\}$ there exist $v_1, v_2 \in U$ for which $\chi(\{v_1, v_2\}) = c_i$. Further, let $w \not\in F$ be a new point, $V \coloneqq  U \cup \{w\}$, and let $\bU$ be the echeloned subspace of $\bF$ induced by $U$. We define an echelon $\leq_\bV$ on $V$ such that 
\begin{itemize}
    \item $\leq_\bV \cap~U^2 =\  \leq_\bU (=\  \leq_\bF \cap~U^2)$,
    \item for all $i\in \{1,\dots,k\}$, for any $v_1,v_2 \in V$ such that $\chi (\{v_1,v_2\})=c_i$, and for all $u\in U_i$, $(w, u) \sim_\bV (v_1, v_2)$,
    \item for any $u_1,u_2 \in U_*$ and $(x, y) \in U^2$ $(x, y) <_\bV (w, u_1)\sim_\bV (w,u_2)$.
\end{itemize}

By the construction $\bU$ is a subspace of $\bV$. Since $\bF$ is universal and weakly homogeneous, there exists an embedding $h \colon \bV \embedsto \bF$ such that the following diagram commutes:

\begin{equation}\tag{\dag}\label{WHstar}\begin{tikzcd}
\bU\arrow[dr, hook, "="'] \arrow[r, hook, "="]& \bV \arrow[d, dashed, hook, "h"]\\
  &  \bF.
\end{tikzcd}
\end{equation}
Let $z \coloneqq  h(w)$. It remains to show that $\chi(\{z, u\}) = c_i$, whenever $u \in U_i$ for any $i \in \{1, \dots, k\} \cup \{*\}$. 

First, fix $i \in \{1, \dots, k\}$ and $u \in U_i$. Recall that there exist some $v_1, v_2 \in U$ such that $\chi(\{v_1, v_2\})=c_i$. In particular, we have $\eta_\bF(v_1, v_2) = \chi(\{v_1, v_2\}) =c_i$. By construction we have $(w,u) \sim_\bV (v_1, v_2)$. Thus, 
    $(h(w), h(u)) \sim_\bF (h(v_1), h(v_2))$. By the commutativity of diagram~\eqref{WHstar} we obtain that $(z,u)\sim_\bF (v_1,v_2)$.
As a result $\chi(\{z, u\}) = \eta_\bF(z, u) = \eta_\bF(v_1, v_2) = c_i$.

It remains to show that $\chi(\{z, u\}) = c_\ast$ for all $u \in U_*$. Let $u \in U_*$ be arbitrary. Then by definition of $\leq_\bV$ we know that for $i \in \{1, \dots, k\}$ and any $(v_1, v_2) \in U^2$ with $\chi(\{v_1, v_2\}) = c_i$, we have $(v_1, v_2) <_\bV (w, u)$. We get that $(h(v_1), h(v_2)) <_\bF (h(w), h(u))$, and hence, by the commutativity of diagram~\eqref{WHstar}, $(v_1, v_2) <_\bF (z, u)$. In particular, $\eta_\bF(z, u) \not= \eta_\bF(v_1, v_2) = c_i$ for any $i \in \{1, \dots, k\}$. Consequently, $\chi(\{z, u\}) =  c_\ast$.
\end{proof}

\begin{COR}
For any $c \in E(\bF) \setminus \{\bot_\bF\}$ the graph $(F, E)$ with the set of edges defined by
\[\{u, v\} \in E \quad :\Longleftrightarrow \quad \eta_\bF(u, v) = c\]
is isomorphic to the Rado graph.
\end{COR}

\begin{proof}
Fix a $c \in E(\bF) \setminus \{\bot_\bF\}$ and pick any $c_\ast \not\in E(\bF)$. Define
\[
\chi \colon [F]^2 \to \{c, c_\ast\},\, \{u, v\} \mapsto \begin{cases}
c& \text{if } \eta_\bF(u, v) = c,\\
c_\ast& \text{otherwise}.
\end{cases}
\]
By Proposition~\ref{C-coloured_G}, $(F, \chi) \cong \Tc$, with $C = \{c, c_\ast\}$. Consequently, by Observation~\ref{*2=Rado}  $(F, E)$  is then isomorphic to the Rado graph.
\end{proof}

\begin{THM}\label{prop_F-1isC-colHomUniGraph}
Fix $C \coloneqq  E(\bF) \setminus \{\bot_\bF\}$. Then, the $C$-coloured graph $(F, \chi)$ with  
\[
\chi \colon [F]^2 \to C,\quad \{u, v\} \mapsto  \eta_\bF(u, v)
\]
is isomorphic to $\Tc$.
\end{THM}

\begin{proof}
Fix a positive integer $k \geq 2$. Pick any $k$ colours from $C$ and enumerate them as $\{c_1, c_2, \dots, c_k\}$. Also, choose any $k$ pairwise disjoint finite sets of points $U_1, U_2, \dots, U_k$ from $F$.
Let $c_{k+1}\in C\setminus\{c_1,\dots, c_k\}$. Define $C'\coloneqq \{c_1,\dots,c_k,c_{k+1}\}$ and a colouring
\[
\chi_{C'} \colon [F]^2 \to C',\quad \{u,v\} \mapsto \begin{cases}
c_i& \text{if } \eta_\bF(u, v) = c_i, i \in \{1, 2, \dots, k\},\\
c_{k+1}& \text{otherwise}.
\end{cases}
\]
By Proposition~\ref{C-coloured_G}, $(F, \chi_{C'})$ is isomorphic to $\Tcc$. In particular, it enjoys the $(\ast_{k})$-property. Therefore, there exists a vertex $z \in F$ such that for any $i \in \{1, 2, \dots, k\}$, $\chi_{C'}(\{u, z\}) =\chi (\{u,z\})= c_i$ for all $u \in U_i$. Owing to the arbitrary choice of $U_i$'s, we get that $(F, \chi)$ itself has the $(\ast_k)$-property. Due to the arbitrary choice of $k$, $(F, \chi)$ has the $(\ast_\infty)$-property and so by Proposition~\ref{prop_*inf} it is indeed isomorphic to $\Tc$.
\end{proof}

From this point on, $\Tf$ will always stand for the $C$-coloured graph described in Theorem~\ref{prop_F-1isC-colHomUniGraph}. 

\begin{COR}
    Let $\bH$ be a finite echeloned space. Let $\bot_\bH<c_1<\cdots<c_k$ be an enumeration of $E(\bH)$. Then for all $d_1<\cdots<d_k\in E(\bF) \setminus \{\bot_\bF\}$ there exists an embedding $\iota\colon H\embedsto F$ such that $\hat{\iota}(c_i)=d_i$ for all $i\in \{1,\dots,k\}$.
\end{COR}

\begin{proof}
    Let $C \coloneqq  E(\bF) \setminus \{\bot_\bF\}$. Let $\Gamma$ be the $C$-coloured graph $(H,\chi)$ with $\chi (\{u,v\})=d_i$ if $\eta_{\bH}(u,v)=c_i$. By the universality of $\Tf$ established in Theorem~\ref{prop_F-1isC-colHomUniGraph}, there exists an embedding $\kappa\colon \Gamma\embedsto \Tf$ of $C$-coloured graphs. It induces an embedding $\iota\colon  \bH\embedsto \bF$ with $\hat{\iota}(c_i)=d_i$.
\end{proof}
Below, we show that universal homogeneous $C$-coloured graphs, for countable $C$, can be constructed probabilistically. 

\begin{PROP}\label{randConst}
Let $C = \{c_1, c_2, \dots\}$ be a countable set of colours, let $V$ be a countably infinite set, and let  $p \in (0,1)$. Let $\chi\colon [V]^2\to C$ be a random colouring that assigns independently the colour $c_i$ with probability $(1-p)^{i-1}p$, \ie 
\[
\forall u,v,u\neq v\colon P[\chi (\{u,v\})=c_i]=(1-p)^{i-1}p.
\] 
Then, with probability 1, the graph $(V,\chi)$ is isomorphic to $\Tc$.
\end{PROP}

\begin{proof}
We show that $(V,\chi)$ has the $(\ast_\infty)$-property with probability $1$. Thus, we fix some integer $k \geq 2$.

Let $U_1,\dots,U_k$ be disjoint finite subsets of $V$. Let $c_{i_1},\dots,c_{i_k}\in C$. Take arbitrary $z\in V$. We compute the probability that $z$ is not eligible for the $(\ast_\infty)$-property, \ie
\[
P\left[\exists j\in \{1,\dots, k\}\; \exists u\in U_j\colon  \chi(\{u,z\})\neq c_{i_j}\right]=
 1 - (1-p)^{\sum\limits_{\ell = 1}^{k}({i_\ell}-1)|U_\ell|} \cdot p^{\sum\limits_{\ell = 1}^{k}|U_\ell|} <1.
\]
Since edges are coloured independently, we have 
\[
P[\forall z\in V\;\exists j\in \{1,\dots, k\}\; \exists u\in U_j\colon  \chi(\{u,z\})\neq c_{i_j}]=0.
\]
There are only countably many choices for $U_i$ and $c_{i_j}$, so 
\[
P[\text{$(\ast_\infty)$-property fails}]=0.
\]

Hence, our observed graph is isomorphic to $\Tc$ (with probability $1$).
\end{proof}

\section{The Ramsey property}\label{sec_Ramsey}

Let $\calK$ be a class of relational structures. For $\bA, \bB \in \calK$ denote by $\binom{\bB}{\bA}$ the set of all substructures of $\bB$ isomorphic to $\bA$. Then a class $\calC$ is a \emph{Ramsey class} if for every two structures $\bA \in \calC$ and $\bB \in \calC$ and every $k \in \NN^+$ there exists a structure $\bC \in \calC$ such that the following holds: For every partition of $\binom{\bC}{\bA}$ into $k$ classes there is $\tilde{\bB} \in \binom{\bC}{\bB}$ such that $\binom{\tilde{\bB}}{\bA}$ belongs to a single class of the partition. A countable relational structure $\calF$ has the \emph{Ramsey property} if $\Age(\calF)$ is a Ramsey class. 

Naturally, the question of whether or not the class of finite echeloned spaces is a Ramsey class arises. Homogeneous structures do not necessarily have the Ramsey property, the neccessary condition for the latter is that the corresponding automorphism group preserves a linear order.
In fact, often times this condition is the only obstruction, as will be the case here.

We approach this problem with the tools of topological dynamics. Let $G$ be a \emph{topological group}, \ie a group equipped with a topology  with respect to which both, multiplication and inverse map are continuous functions. A \emph{$G$-flow} is a continuous action $G \times X \to X$, where $X$ is a nonempty compact Hausdorff space. We say that $G$ is \emph{extremely amenable} if every $G$-flow has a fixed point.  The connection between extreme amenability and the Ramsey property is made precise in the following theorem. For its statement we will need to recall that a structure $\bA$ whose signature contains a distinguished binary relation symbol $\preceq$ is said to be \emph{ordered} if  $\preceq_\bA$ is a linear order.

\begin{THM}[Kechris, Pestov, Todor\v cevi\' c~\cite{kechris2003fraisse}]\label{thm_KPT_EA=RP}
Let $\calK$ be a \Fraisse class of ordered structures and let $\bA$ be its \Fraisse limit. Then $\Aut(\bA)$ is extremely amenable if and only if $\calK$ has the Ramsey property.
\end{THM} 

Now, we begin our exposition by observing that the key feature of a \Fraisse class, namely the amalgamation property, holds for our class of interest.   

\begin{PROP}
The classes of ordered finite echeloned spaces and of ordered finite $C$-coloured graphs (for countable $C$) are \Fraisse classes.
\end{PROP}
\begin{proof}
The classes of finite echeloned spaces, finite $C$-colored graphs, and finite linear orders are all \Fraisse classes with the strong amalgamation property (for the case of finite echeloned spaces, see  Remark~\ref{echelonhasSAP}). This immediately implies the result (see \cite[Section 3.9, p.\ 59]{Cam1990}). 
\end{proof}

 The \Fraisse limit of the class of finite ordered echeloned spaces is actually obtained from $\bF$ by the addition of an appropriate linear order $\preceq_\bF$ isomorphic to the natural order on $\QQ$. We will denote this \Fraisse limit by $(\bF, \preceq_\bF)$. 

\begin{LEM}\label{lem_orderedversion}
 Let $C\coloneqq E(\bF)\setminus \{ \bot_\bF\}$. Then $(\Tf,\preceq_\bF)$ is a universal homogeneous ordered $C$-coloured graph.
\end{LEM}
\begin{proof}
    Let $(\bH,\preceq_\bH)$ be the ordered echeloned space whose echelon $\leq_\bH$ is defined by 
    \[
    (x,y)\leq_\bH (u,v) \quad :\Longleftrightarrow \quad  \chi_\bH(\{x,y\})\leq \chi_\bH(\{u,v\}),    
    \]
    from the edge coloring $\chi_\bH\colon[H]^2\to C$ of a countable homogeneous ordered $C$-coloured graph $(H,\chi_\bH,\preceq_\bH)$.
    We aim to show that $(\bH, \preceq_\bH)$ is a universal homogeneous ordered echeloned space. 
    To this end it suffices to show that $(\bH,\preceq_\bH)$ is universal and weakly homogeneous. Let $(\bA,\preceq_\bA)$ be a finite ordered echeloned subspace of $(\bH,\preceq_\bH)$, let $(\bB,\preceq_\bB)$ be a finite ordered echeloned space, and let $\iota\colon (\bA,\preceq_\bA)\embedsto(\bB,\preceq_\bB)$ be an embedding. Note that $\kappa\colon E(\bA)\setminus\{\bot_\bA\}\to C,\, [(x,y)]_{\sim_\bA}\mapsto \chi_\bH(\{x,y\})$ is an order embedding. The same holds for $\hat\iota\colon E(\bA)\embedsto E(\bB)$. Let $\tilde\iota\colon E(\bA)\setminus\{\bot_\bA\}\embedsto E(\bB)\setminus\{\bot_\bB\}$ be the appropriate restriction of $\hat\iota$. Recall that $C$ is isomorphic to $\QQ$. In other words, it is a universal homogeneous chain. Thus, there exists an order embedding $\tilde\kappa\colon E(\bB)\setminus\{\bot_\bB\}\embedsto C$ that makes the following diagram commutative:  
    \[
    \begin{tikzcd}
        E(\bA)\setminus\{\bot_\bA\} \arrow[r, hook, "\tilde\iota"] \arrow[dr, hook, "\kappa"'] &  E(\bB)\setminus\{\bot_\bB\}\arrow[d, dashed, hook, "\tilde\kappa"]\\
        & C.
    \end{tikzcd}
    \]
    Next we define 
    \begin{align*}
        \chi_\bA&\colon [A]^2\to C,  & \{x,y\} &\mapsto \kappa(\eta_\bA(x,y)) = \chi_\bH(\{x,y\}),\\
        \chi_\bB&\colon [B]^2\to C, & \{u,v\}&\mapsto \tilde\kappa(\eta_\bB(u,v)).
    \end{align*}
    Then $\iota\colon (A,\chi_\bA,\preceq_\bA)\to(B,\chi_\bB,\preceq_\bB)$ is an embedding of ordered $C$-coloured graphs. Indeed, $\iota$ preserves $\preceq$, and
    \begin{multline*}
    \forall \{x,y\}\in[A]^2\,:\, \chi_\bB(\{\iota(x),\iota(y)\}) = \tilde\kappa(\eta_\bB(\iota(x),\iota(y)))\\
    = \tilde\kappa(\tilde\iota(\eta_\bA(x,y)))
    = \kappa(\eta_\bA(x,y)) = \chi_\bA(\{x,y\}).
    \end{multline*}
    Next note that $(A,\chi_\bA,\preceq_\bA)$ is an ordered $C$-coloured subgraph of $(H,\chi_\bH,\preceq_\bH)$. Indeed, $A\subseteq H$, ${\preceq_\bA}\subseteq {\preceq_\bH}$, and 
    \[
    \forall\{x,y\}\in[A]^2\,:\, \chi_\bA(\{x,y\}) = \kappa(\eta_\bH(x,y)) = \chi_\bH(\{x,y\}),
    \]
    by the definition of $\kappa$. 

    Since $(H,\chi_\bH,\preceq_\bH)$ is universal and weakly homogeneous, there exists $\epsilon\colon (B,\chi_\bB,\preceq_\bB)\embedsto(H,\chi_\bH,\preceq_\bH)$, such that the following diagram commutes:
    \begin{equation}\label{***}
    \begin{tikzcd}
        (A,\chi_\bA,\preceq_\bA) \arrow[r, hook, "\iota"] \arrow[dr, hook, "="'] &  (B,\chi_\bB,\preceq_\bB)\arrow[d, dashed, hook, "\epsilon"]\\
        & (H,\chi_\bH,\preceq_\bH).
    \end{tikzcd}
    \end{equation}
    We claim that $\epsilon\colon(\bB,\preceq_\bB)\to(\bH,\preceq_\bH)$ is an embedding. Clearly, $\epsilon\colon(B,\preceq_\bB)\embedsto(H,\preceq_\bH)$ is an order embedding. 
    So let $(x,y)\leq_\bB(u,v)$. 
    If $x=y$, then $\epsilon(x)=\epsilon(y)$, and thus $(\epsilon(x),\epsilon(y))\leq_\bH(\epsilon(u),\epsilon(v))$. If on the other hand $x\neq y$ then  we compute 
    \[\arraycolsep=1.4pt\def\arraystretch{1.3}
    \begin{array}{rclcclc}
        (x,y) & \leq_\bB & (u,v) & \iff & \eta_\bB(x,y) & \leq & \eta_\bB(u,v) \\
        & & & \overset{\tilde\kappa\text{ emb.}}{\iff} 
        & \tilde\kappa(\eta_\bB(x,y)) & \leq & \tilde\kappa(\eta_\bB(u,v))\\
        & & & & \hfill \rotatebox{90}{=} \hfill  & & \hfill \rotatebox{90}{=} \hfill \\
        & & & \iff 
        & \chi_\bB(\{x,y\}) & \leq & \chi_\bB(\{u,v\})\\
        & & & & \hfill \rotatebox{90}{=} \hfill  & & \hfill \rotatebox{90}{=} \hfill \\
        & & & \overset{\epsilon\text{ emb.}}{\iff} 
        & \chi_\bH(\{\epsilon(x),\epsilon(y)\}) & \leq & \chi_\bH(\{\epsilon(u),\epsilon(v)\})\\
        & & & \overset{\text{def.}}{\iff} & (\epsilon(x),\epsilon(y)) & \leq_\bH &  (\epsilon(u),\epsilon(v)).        
    \end{array}
    \]
    From \eqref{***} we obtain that 
        \begin{equation*} %\label{**}
    \begin{tikzcd}
        (\bA,\preceq_\bA) \arrow[r, hook, "\iota"] \arrow[dr, hook, "="'] &  (\bB,\preceq_\bB)\arrow[d, dashed, hook, "\epsilon"]\\
        & (\bH,\preceq_\bH)
    \end{tikzcd}
    \end{equation*}
 commutes. This shows that $(\bH,\preceq_\bH)$ is universal and weakly homogeneous. 

    By the uniqueness of \Fraisse limits there exists an isomorphism $\phi\colon(\bF,\preceq_\bF)\to(\bH,\preceq_\bH)$. Since 
    \[
    \begin{tikzcd}
        F^2 \arrow[r,"\eta_\bF"] \arrow[d,"\phi^2"']& E(\bF)\arrow[d,"\hat\phi"]\\
        H^2 \arrow[r,"\eta_\bH"]& E(\bH)
    \end{tikzcd}
    \]
    commutes, by the definition of $\chi_\bF$ and $\chi_\bH$, also the following diagram commutes:
    \[
    \begin{tikzcd}
        {[F]}^2 \arrow[r,"\chi_\bF"] \arrow[d,"{[\phi]}^2"']& C\arrow[d,"\tilde\phi"]\\
        {[H]}^2 \arrow[r,"\chi_\bH"]& C,
    \end{tikzcd}
    \]
    where $[\phi]^2\colon\{x,y\}\mapsto\{\phi(x),\phi(y)\}$, and where $\tilde\phi$ is the appropriate restriction of $\hat\phi$ to $C$. 

    However, this means that $(\Tf,\preceq_\bF)$ and $(H,\chi_\bH,\preceq_\bH)$ are practically equal, up to names of vertices and names of colours. In other words, $(\Tf,\preceq_\bF)$ is a universal homogeneous $C$-coloured graph, too. 
\end{proof}

\begin{LEM} \label{kernelextramen}
    $\Aut(\Tf, \preceq_\bF)$ is extremely amenable.
\end{LEM}
\begin{proof} 
By Theorem~\ref{thm_KPT_EA=RP}, it suffices to prove that the class of finite ordered echeloned spaces has the Ramsey property. In order to show this, we consider the class of all finite ordered $C$-coloured graphs, where $C\coloneqq  E(\bF)\setminus\{\bot_{\bF}\}$. Let $c\in C$. 

Note that there exists a canonical bijection $\Phi$ between the class of all ordered $C$-coloured graphs and all ordered $C\setminus\{c\}$-edge-coloured (simple) graphs such that, for all ordered $C$-coloured graphs $\bA$ and $\bB$,
\begin{enumerate}[label=(\roman*), ref=(\roman*)]
    \item $\Phi(\bA)$ has the same vertex set as $\bA$, and
    \item the set of embeddings from $\bA$ to $\bB$ coincides with the set of embeddings from $\Phi(\bA)$ to $\Phi(\bB)$.
\end{enumerate}
Clearly, the class of finite $C\setminus\{c\}$-edge-coloured graphs has the free amalgamation property, thus the class of finite ordered $C\setminus\{c\}$-edge-coloured graphs has the Ramsey property, by a result of Hubi\v cka and Ne\v setril~\cite[Corollary 4.2, p.\ 50]{hubivcka2019all}. 

Thanks to the properties of $\Phi$ the class of finite ordered $C$-coloured graphs has the Ramsey property, too. So $\Aut (\Tf,\preceq_{\bF})$ is extremely amenable due to Theorem~\ref{thm_KPT_EA=RP}.
\end{proof}

For the notation used in the following lemma, see Corollary~\ref{cor_sim_EMB}.
\begin{LEM}\label{lem_AutOfFs}
Let $\beta$ be a local isomorphism of $E(\bF)\setminus \{\bot_{\bF}\}$, with $T \coloneqq  \dom \beta$. Then, there exists an automorphism $\alpha$ of $(\bF, \preceq_\bF)$ such that $\hat{\alpha}\restr_{T} = \beta$.
\end{LEM}
\begin{proof} 
Recall from Theorem~\ref{prop_F-1isC-colHomUniGraph} that $\Tf=(F,\chi)$ is a universal homogeneous $C$-coloured graph, where $C\coloneqq  E(\bF)\setminus\{\bot_{\bF}\}$. Denote the elements of $\dom\beta$ by $c_1,\dots,c_k$ and let $c_{k+i}\coloneqq \beta(c_i)$ for each $i\in \{1,\dots,k\}$. Let $c_\ast\in C$ be strictly greater than any element of $\{c_1,\dots,c_k,c_{k+1},\dots,c_{2k}\}$. Consider the $C$-coloured graph $\Delta=(V,\chi_{\Delta})$ given by 
\begin{itemize}
    \item $D=\{v_1,\dots,v_{2k}\}\dotcup \{w_1,\dots,w_{2k}\}$, and
    \item $\chi_{\Delta}(\{x,y\})=\begin{cases}
        c_i & \text{if } \{x,y\}=\{v_i,w_i\} \text{ for some } i\in\{1,\dots,2k\},\\
        c_\ast & \text{else}.
    \end{cases}$
\end{itemize}
Since $\Tf$ is universal, we may assume without loss of generality that $D\subseteq F$ and $\chi_\Delta= \chi\restr_{[D]^2}$. Consider the superstructure $\tilde{\Delta}=(\tilde{D},\chi_{\tilde{\Delta}})$ of $\Delta$ given by 
\begin{itemize}
    \item $\tilde{D}=D\dotcup\{\tilde{v}_1,\dots,\tilde{v}_{k}\}\dotcup \{\tilde{w}_1,\dots,\tilde{w}_{k}\}$, and
    \item ${\chi}_{\tilde{\Delta}}(\{x,y\})=\begin{cases}
        \chi_\Delta(\{x,y\}) & \text{if } x,y\in D,\\
        c_{k+i} & \text{if }\{x,y\}=\{\tilde{v}_i,\tilde{w}_i\} \text{ for some } i\in\{1,\dots,k\},\\
        c_\ast & \text{else}.
    \end{cases}$
\end{itemize}
Then $\Delta$ is a substructure of both $\Tf$ and  $\tilde\Delta$. Since $\Tf$ is universal and weakly homogeneous, there exists an embedding $\kappa\colon\tilde\Delta\embedsto\Tf$ such that the following diagram commutes:
\[\begin{tikzcd}
\Delta\arrow[dr, hook, "="'] \arrow[r, hook, "="] &  \tilde\Delta\arrow[d, dashed, hook, "\kappa"]\\
& \Tf.
\end{tikzcd}
\]
Thus, again, we may assume without loss of generality that $\tilde{D}\subseteq F$ and $\chi_{\tilde\Delta}=\chi\restr_{[\tilde{D}]^2}$.

The elements of $D$ are linearly ordered  by $\preceq_{\Delta}\coloneqq\preceq_\bF\cap V^2$.  Clearly, $\preceq_{\Delta}$ can be extended to $\tilde{D}$ in such a way that the mapping 
\[h\colon D_0\to \tilde{D}\setminus D,\, 
 v_i\mapsto \tilde{v}_i,\, w_i\mapsto \tilde{w}_i, i\in \{1,\dots,k\},\quad\text{where } D_0 = \{v_1,\dots, v_k,w_1,\dots, w_k\},
\]
is monotone. Let us fix such an extension $\preceq_{\tilde{\Delta}}$. Let $\bD_0$, $\bD$ and $\tilde{\bD}$ be the subspaces of $\bF$ induced on $D_0$, $D$, and $\tilde{D}$, respectively. Then $(\bD,\preceq_{\Delta})$ is both, a substructure of $(\tilde{\bD},\preceq_{\tilde{\Delta}})$ and of $(\bF,\preceq_{\bF})$. 
\begin{nCLM}
    $h\colon \bD_0\to\tilde\bD$ is an embedding.
\end{nCLM} 
\begin{proof}[Proof of the claim]
    By Corollary~\ref{cor_sim_EMB} there exist order embeddings  $\gamma_0\colon E(\bD_0)\embedsto E(\bF)$ and $\gamma\colon E(\bD)\embedsto E(\bF)$, such that the following diagrams commute:
    \[
    \begin{tikzcd}
        D_0^2 \ar[d,hook',"="']\ar[r,"\eta_{\bD_0}"]& E(\bD_0) \ar[d,hook',"\gamma_0"]\\
        F^2 \ar[r,"\eta_\bF"] & E(\bF),
    \end{tikzcd}\qquad\qquad 
    \begin{tikzcd}
        D^2 \ar[d,hook',"="']\ar[r,"\eta_{\bD}"]& E(\bD) \ar[d,hook',"\gamma"]\\
        F^2 \ar[r,"\eta_\bF"] & E(\bF).
    \end{tikzcd}
    \]
    Note that the map 
    \[
    \beta\colon\{c_1,\dots,c_k,c_\ast\}\to\{c_1,\dots,c_{2k},c_\ast\},\quad c_i\mapsto c_{k+i},\, c_\ast\mapsto c_\ast,\, i\in\{1,\dots k\}
    \]
    is an order embedding. Thus the map
    \[
    \hat{h}\colon E(\bD_0)\to E(\bD),\quad e\mapsto \gamma^{-1}(\beta(\gamma_0(e))) 
    \]
    is a well-defined order embedding satisfying $\eta_\bD\circ h^2 = \hat{h}\circ\eta_{\bD_0}$. Hence $h\colon \bD_0\to\bD$ is an embedding by Corollary~\ref{cor_sim_EMB}.
\end{proof}

In turn, letting $\preceq_{\Delta_0}\coloneqq\preceq_\Delta\cap D_0^2$, it follows that $h\colon (\bD_0,\preceq_{\Delta_0})\to (\tilde{\bD},\preceq_{\tilde\Delta})$ is an embedding, too. 

Since $(\bF,\preceq_{\bF})$ is universal and weakly homogeneous there exists $\iota\colon (\tilde{\bD},\preceq_{\tilde{\Delta}})\embedsto (\bF,\preceq_{\bF})$ such that the following diagram commutes:
\[\begin{tikzcd}
(\bD,\preceq_{\Delta})\arrow[dr, hook, "="'] \arrow[r, hook, "="] &  (\tilde{\bD},\preceq_{\tilde{\Delta}})\arrow[d, dashed, hook, "\iota"]\\
& (\bF,\preceq_{\bF}).
\end{tikzcd}
\]
Consequently, $\iota\circ h\colon (\bD_0,\preceq_{\bD_0})\to (\bF,\preceq_\bF)$ is an embedding. By the homogeneity of $(\bF,\preceq_{\bF})$, there exists an $\alpha\in \Aut(\bF,\preceq_{\bF})$ that extends $\iota\circ h$. It remains to check that $\hat\alpha$ extends $\beta$. 
To this end take $i\in \{1,\dots,k\}$. From the diagram above we know that $(v_{k+i},w_{k+i})=(\iota(v_{k+i}),\iota(w_{k+i}))$. Moreover, as $(v_{k+i},w_{k+i})\sim_{\tilde{\bD}} (\tilde{v}_i,\tilde{w}_i)$, we have that $(\iota (v_{k+i}),\iota(w_{k+i})) \sim_{\bF} (\iota(\tilde{v_i}),\iota(\tilde{w_i}))$. 

Thus
\begin{align*}
    \hat{\alpha}(c_i)&=[(\alpha(v_i),\alpha(w_i))]_{\sim_{\bF}}=\eta_{\bF}(\alpha(v_i),\alpha(w_i))\\
    &= \eta_{\bF}(\iota(h(v_i)),\iota(h(w_i))) =\eta_{\bF}(\iota(\tilde{v}_i),\iota(\tilde{w}_i))\\
    &=  \chi(\{\iota(\tilde{v}_i),\iota(\tilde{w}_i)\}) = \chi(\{\iota({v}_{k+i}),\iota({w}_{k+i})\}) \\
    &= \chi(\{{v}_{k+i},{w}_{k+i}\}) = \chi_\Delta(\{{v}_{k+i},{w}_{k+i}\})\\
    &= c_{k+i}=\beta(c_i).
\end{align*}
In other words, $\hat{\alpha}\restr_{\dom\beta}=\beta$.
\end{proof}

Prior to stating and proving the main result of this section, we recall a natural family of topological groups. Let $X$ be a set. Then the corresponding full symmetric group $\Sym(X)$, \ie the group of all self-bijections of $X$, together with the topology of pointwise convergence associated with the discrete topology on $X$ is a topological group. In turn, if $G$ is a subgroup of $\Sym(X)$, then $G$, endowed with the relative topology inherited from $\Sym(X)$, is a topological group, and the sets of the form \begin{displaymath}
  V_G(E) \coloneqq  \{g \in G \mid \forall x \in E \colon \, gx=x \} \qquad (X \supseteq E \text{ finite})
\end{displaymath} constitute a neighborhood basis at the identity in $G$. If $X$ is countable, then $\Sym(X)$, as well as any of its closed topological subgroups, is Polish.

The proof of the following result will make use of Theorem~\ref{thm_KPT_EA=RP} along with the persistence of extreme amenability of topological groups under extensions (see~\cite[Corollary~6.2.10]{PestovBook}).

\begin{THM}\label{thm_ShasRamsey}
$(\bF, \preceq_\bF)$ has the Ramsey property.
\end{THM}
\begin{proof} 
Due to Theorem~\ref{thm_KPT_EA=RP} $(\bF, \preceq_\bF)$ has the Ramsey property if and only if $\Aut (\bF, \preceq_\bF)$ is extremely amenable. 

Let $C \coloneqq  E(\bF) \setminus \{\bot_\bF\}$. Let $\Tf=(F,\chi)$ be the $C$-coloured graph constructed in Proposition~\ref{prop_F-1isC-colHomUniGraph}. Recall that $(\bF,\preceq_\bF)$ is the \Fraisse limit of the class of finite ordered echeloned spaces. Therefore $\Aut(\Tf, \preceq_\bF)$ is a subgroup of $\Aut(\bF, \preceq_\bF)$, namely the one of automorphisms that setwise preserve each equivalence class of $\sim_{\bF}$.  Let $\pi \colon \Aut(\bF, \preceq_\bF) \to \Aut(E(\bF)) ,\, \alpha \mapsto \hat{\alpha}$. Clearly, $\ker(\pi) = \Aut(\Tf, \preceq_\bF)$, which is extremely amenable by Lemma~\ref{kernelextramen}.

\begin{CLM}\label{clm_continuous}
${\pi}$ is continuous.
\end{CLM}
\begin{proof}[Proof of the claim]
 It suffices to note that, for every $c \in E(\bF)$, we have  $\pi[V_{\Aut(\bF, \preceq_\bF)}(\{x,y\})] \subseteq V_{\Aut(E(\bF))}(\{c\})$, where  $x,y\in F$ such that $\eta_{\bF}(x,y)=c$. 
\end{proof}

\begin{CLM}\label{clm_open}
${\pi}$ is open onto its image.
\end{CLM}
\begin{proof}[Proof of the claim]
As $\pi$ is a homomorphism, it is enough to show that, for every finite subset $E$ of $F$, there exists a finite subset $\tilde{E}$ of $E(\bF)$ such that
\[\pi[\Aut(\bF, \preceq_\bF)] \cap V_{\Aut(E(\bF))}(\tilde{E}) \subseteq \pi[V_{\Aut(\bF, \preceq_\bF)}(E)].\]
Let $E$ be a finite subset of $F$. Define $\tilde{E} \coloneqq  \eta_\bF[E^2]$. Let $\gamma \in \pi[\Aut(\bF, \preceq_\bF)] \cap V_{\Aut(E(\bF))}(\tilde{E})$. There then exists some $\alpha_{0} \in \Aut(\bF, \preceq_\bF)$ with $\pi(\alpha_{0}) = \gamma$. From $\pi(\alpha_{0}) = \gamma \in V_{\Aut(E(\bF))}(\tilde{E})$ we infer that $E \to \alpha_{0}[E], \, x \mapsto \alpha_{0}(x)$ is a local isomorphism in $(\Tf,\preceq_\bF)$. Since $(\Tf,\preceq_\bF)$ is homogeneous by Lemma~\ref{lem_orderedversion}, there exists $\alpha_{1} \in \Aut(\Tf,\preceq_\bF)$ such that $\alpha_{1}\restr_{E} = \alpha_{0}\restr_{E}$. We see that $\beta \coloneqq  \alpha_{1}^{-1}\circ \alpha_{0} \in V_{\Aut(\bF, \preceq_\bF)}(E)$. Moreover, as $\alpha_{1} \in \Aut(\Tf,\preceq_\bF)$, one has $\pi(\alpha_{1}) = \id_{E(\bF)}$, so that $\pi(\beta) = \pi(\alpha_{1})^{-1}\pi(\alpha_{0}) = \pi(\alpha_{0}) = \gamma$. Hence, $\gamma = \pi(\beta) \in \pi[V_{\Aut(\bF, \preceq_\bF)}(E)]$ as desired.
\end{proof}

\begin{CLM}
$\pi$ is surjective. 
\end{CLM}

\begin{proof}[Proof of the claim]
We first observe that $\pi[\Aut(\bF, \preceq_\bF)]$ is dense in $\Aut(E(\bF))$ with respect to the topology of pointwise convergence. Indeed, for any $\beta\in \Aut(E(\bF))$ and any finite subset $E_0$ of $E(\bF)$, there exists an $\alpha \in \Aut(\bF, \preceq_\bF)$ such that $\pi(\alpha)\restr_{E_0} = \hat{\alpha} \restr_{E_0} = \beta \restr_{E_0}$, thanks to Lemma~\ref{lem_AutOfFs}. 

Now, by Claims~\ref{clm_continuous} and~\ref{clm_open}, the topological subgroup $\pi[\Aut(\bF, \preceq_\bF)]$ of $\Aut(E(\bF))$ is actually isomorphic to the quotient of the Polish group $\Aut(\bF, \preceq_\bF)$ by the closed normal subgroup $\ker(\pi)$. It is therefore itself a Polish group (see, \eg \cite[Proposition~1.2.3]{BeckerKechris}) and thus closed in $\Aut(E(\bF))$ (see, \eg \cite[Proposition~1.2.1]{BeckerKechris}). Hence, it is equal to $\Aut(E(\bF))$, \ie $\pi$ is surjective.

(Note that the previous argument does not actually rely on the separability assumption behind the definition of a Polish space. Indeed, the quotient of a metrizable group, complete for its upper uniform structure, is again complete for its upper uniform structure~\cite[Corollary 2, p.~27]{Brown-1971} (see also~\cite[Theorem~2]{Brown-1972}) and therefore closed in any group in which it topologically embeds~\cite[3.24]{RoelckeDierolf}.)
\end{proof}

It follows by these three claims that the group $\Aut(\bF, \preceq_\bF)$ has an extremely amenable closed normal subgroup $\ker(\pi) = \Aut(\Tf, \preceq_\bF)$ whose corresponding quotient $\faktor{\Aut(\bF, \preceq_\bF)}{\ker \pi}$, being isomorphic to $\Aut(\QQ, <)$ by Lemma~\ref{distQ}, is also extremely amenable \cite{pestov1998free}. Hence $\Aut(\bF, \preceq_\bF)$ itself is extremely amenable (see, \eg \cite[Corollary~6.2.10]{PestovBook}) and therefore $(\bF, \preceq_\bF)$ has the Ramsey property by Theorem~\ref{thm_KPT_EA=RP}.
\end{proof}

\section{Universality of \texorpdfstring{$\Aut(\bF)$}{Aut(F)}}\label{sec_Katetov}
	The goal of this section is to prove the following universality property for the automorphism group of the countable universal homogeneous echeloned space $\bF$:
	\begin{THM}\label{AutUniversal}
		The full symmetric group $\Sym(\NN)$ topologically embeds into $\Aut(\bF)$ (with respect to the  pointwise convergence topology). 
	\end{THM}
	For a proof of this claim we are going to employ the theory of \Katetov functors in the sense of \cite{KubMas16}. If we succeed to equip the class of finite echeloned spaces with a \Katetov functor, then  from \cite[Corollary 3.9, Corollary 3.12]{KubMas16} it follows that the automorphism group of every countable echeloned space topologically embeds into $\Aut(\bF)$ (with respect to the topology of pointwise convergence). Adding to this the observation that the unique echeloned space on $\NN$ with two-element echeloning has  automorphism group $\Sym(\NN)$, the claim of  Theorem~\ref{AutUniversal} follows readily.

    Note that Theorem~\ref{AutUniversal} could be stated stronger. When looking into the details of the \Katetov construction it becomes apparent that actually the natural action of $\Sym(\NN)$ is, up to action isomorphism, a subaction of the natural action of $\Aut(\bF)$ on $F$ (cf.~\cite[Theorems 2.2, 3.3]{KubMas16}). 
 
	The rest of this section is devoted to the proof that finite echeloned spaces admit a \Katetov functor. 
	\begin{DEF}
		Let $\calC$ be an age, \ie
		\begin{itemize}
			\item $\calC$ is a class of finitely generated structures of the same type,
			\item $\calC$ is isomorphism-closed,
			\item $\calC$ is closed under taking substructures (it has the hereditary property),
			\item $\calC$ has the joint embedding property,
			\item $\calC$ splits into countably many isomorphism classes.
		\end{itemize}
		Let $\cat{C}$ be a category whose object class consists of all those countably generated structures $\bX$ with the property that all finitely generated substructures of $\bX$ are in $\calC$, and whose morphism class contains all embeddings. Let $\cat{A}$ be the full subcategory of $\cat{C}$ induced by $\calC$. 
		
		A functor $K\colon\cat{A}\to\cat{C}$ is called a \emph{\Katetov functor} if 
		\begin{itemize}
			\item $K$ preserves embeddings,
			\item there exists a natural embedding $\lambda\colon \Id\embedsto K$ such that for all $\bA\in\calC$ and for all one-point extensions $\bB$ of $\bA$ in $\calC$ there exists an embedding $g\colon \bB\embedsto K(\bA)$ such that the following diagram commutes:
			 \[
			 \begin{tikzcd}
			 	\bA \arrow[d,hook',"="']\arrow[dr,hook',"\lambda_{\bA}"]\\
			 	\bB \arrow[r,hook,"g"] & K(\bA).
			 \end{tikzcd}
			 \]
		\end{itemize} 
	\end{DEF}
	Our particular setting is the following:
	\begin{itemize}
		\item $\calC$ is the class of all finite echeloned spaces,
		\item $\cat{C}$ is the category of countable echeloned spaces with embeddings,
		\item $\cat{A}$ is the full subcategory of $\cat{C}$ induced by $\calC$.
	\end{itemize}
	 Let $\bX=(X,\leq_{\bX})$ be a finite echeloned space. Suppose that $E(\bX)=\{\bot_{\bX},c_1,\dots,c_n\}$, and that 
	\[
		\bot_{\bX} <_{E(\bX)} c_1 <_{E(\bX)} c_2 <_{E(\bX)}\dots <_{E(\bX)} c_n.
	\]
	Let $N_X\coloneqq\{1,\dots,|X|\}$. Let us define a new chain $C_\bX=(C_\bX,\leq_{C_\bX})$ according to 
	\[
	C_\bX \coloneqq E(\bX)\dotcup\{b_\bX\}\dotcup (N_X\times\{0,\dots,n\}),
	\]
	and
	\[ x <_{C_\bX} y  \,:\Longleftrightarrow\, 
	\begin{cases}
		x = \bot_\bX, y\in C_\bX\setminus\{\bot_\bX\}\text{, or}\\	
		x = b_\bX, y\in C_\bX\setminus\{\bot_\bX,b_\bX\}\text{, or}\\
		x = c_i, y= c_j, i<j\text{, or}\\
		x = c_i, y = (k,j), i\leq j, k\in N_X\text{, or}\\
		x = (k,i), y= c_j, i<j, k\in N_X\text{, or}\\
		x = (k,i), y = (\ell,j), i<j, \text{ or } i=j \text{ and } k<\ell, \text{ where }k,\ell\in N_X.
	\end{cases}
	\]
	In other words, $C_\bX$ can be expressed as an ordinal sum of chains as follows:
	\[
		\{\bot_\bX\} \oplus \{b_\bX\} \oplus (N_X\times\{0\}) \oplus \{c_1\} \oplus (N_X\times\{1\}) \oplus \dots \oplus \{c_n\} \oplus (N_X\times\{n\}).
	\]
	Let now $C_\bX^{(X)}\coloneqq\{h\colon X\to C_\bX\mid \bot_\bX\notin\im h\}$, and define 
 \[
 \tilde\eta_\bX\colon(C_\bX^{(X)}\dotcup X)^2\to C_\bX,\quad
	 (x,y)\mapsto\begin{cases}
		\bot_\bX & \text{if } x=y,\\
		\eta_\bX(x,y) & \text{if } x, y \in X, x\neq y,\\
		y(x) & \text{if } x\in X, y\in C_\bX^{(X)},\\
		x(y) & \text{if } x\in C_\bX^{(X)},y\in X,\\
		b_\bX & \text{else.}
	\end{cases}
\]
	Finally we define $K(\bX)\coloneqq (C_\bX^{(X)}\dotcup X,\leq_{K(\bX)})$, where
	\[
	(x,y)\leq_{K(\bX)} (u,v) \quad :\Longleftrightarrow \quad \tilde\eta_\bX(x,y) \leq_{C_\bX} \tilde\eta_\bX(u,v).
	\]
	Clearly, $K(\bX)$ is an echeloned space extending $\bX$. Moreover, $C_\bX\cong E(K(\bX))$. Denote by $\zeta_\bX$ the unique isomorphism from $C_\bX$ to $E(K(\bX))$. Then, in particular, the following diagram commutes:
	\begin{equation}\label{heart}
		\begin{tikzcd}
			(C_\bX^{(X)}\dotcup X)^2 \arrow[r,"\tilde\eta_\bX"]\arrow[dr,"\eta_{K(\bX)}"'] & C_\bX\arrow[d,"\zeta_\bX"]\\
			& E(K(\bX)).
		\end{tikzcd}
		\end{equation}
		In order to make a functor out of $K$, we need to define its action on morphisms. In addition to $\bX$ let us consider another finite echeloned space $\bY$. Suppose that $E(\bY)= \{\bot_\bY, d_1,\dots,d_m\}$ where
  \[
    \bot_\bY<_{E(\bY)} d_1<_{E(\bY)} \dots <_{E(\bY)} d_m.
  \] 
   Let $\phi\colon \bX\embedsto\bY$ be an embedding, \ie the following diagram commutes:
		\begin{equation}\label{star}
		\begin{tikzcd}
			X^2 \ar[r,"\eta_\bX"]\ar[d,hook',"\phi^2"'] & E(\bX) \ar[d,hook',"\hat\phi"]\\
			Y^2 \ar[r,"\eta_\bY"]& E(\bY). 
		\end{tikzcd}
		\end{equation}
    Next we define 
    \[
    \tilde\psi\colon C_\bX\to C_\bY,\quad 
		 x\mapsto \begin{cases}
			\bot_\bY & \text{if } x=\bot_\bX,\\
			b_\bY & \text{if } x=b_\bX,\\
			(k,0) & \text{if } x=(k,0), \, k\in N_X,\\
			d_j & \text{if } x=c_i\text{ and }\hat\phi(c_i)=d_j, i\in\{1,\dots,n\},\, j\in\{1,\dots,m\},\\
			(k,j) & \text{if } x=(k,i), \hat\phi(c_i)=d_j, k\in N_X, i\in\{1,\dots,n\},\, j\in\{1,\dots,m\}.
		\end{cases}
		\]
		Clearly, $\tilde\psi$ is an order embedding. Next define 
  \[
  \psi\colon K(\bX)\to K(\bY),\quad 
		x\mapsto\begin{cases}
			\phi(x) & \text{if } x\in X,\\
			\psi(h) & \text{if } h\in C_\bX^{(X)},
		\end{cases}
		\]
	where 
 \begin{equation}\label{psi}
 \psi(h)\colon Y\to C_\bY,\quad  
		y\mapsto \begin{cases}
			\tilde\psi(h(x)) & \text{if } y=\phi(x),\\
			b_\bY & \text{else.}
		\end{cases}
    \end{equation}
		 Figure~\ref{figpsi} illustrates these definitions.
	\begin{figure}
	\begin{tikzpicture}[mycontainer/.style={draw=gray, inner sep=1ex}]
		\begin{scope}
			\draw[fill=black] %
				(0,0) circle (3pt) node (a) {}
				(1,0) circle (3pt) node (b) {}
				(2,0) circle (3pt) node (c) {}
				(3,0) circle (3pt) node (d) {};
			\node (Z) at (2,0.3) {};
			\node (X) at (0.51,-0.7) {$X$};
			\node[fit=(a)(b)(c)(d)(X)(Z), draw] (boxX) {};
			\draw[fill=black] (1.5,3.5) circle (3pt) node (hh){} node[above=3pt] (h) {$h$};
			\draw  (a) -- (hh) node [above,pos=0.3,sloped] {$\scriptstyle{b_\bX}$};
			\draw  (b) -- (hh) node [above,pos = 0.3,sloped] {$\scriptstyle{(k,0)}$};
			\draw  (c) -- (hh) node [above,pos = 0.3,sloped,rotate=180] {$\scriptstyle{c_i}$};
			\draw  (d) -- (hh) node [above,pos = 0.3,sloped,rotate=180] {$\scriptstyle{(k,i)}$};
			\node (anc1) at (2.5,2) {};
		\end{scope}
\begin{scope}[xshift=60mm]
			\draw[fill=black] %
				(0,0) circle (3pt) node (aa) {}
				(1,0) circle (3pt) node (bb) {}
				(2,0) circle (3pt) node (cc) {}
				(3,0) circle (3pt) node (dd) {}
				(4,0) circle (3pt) node (ee) {};
			\node (end) at (5,0) {};
			\node (pZ) at (2,0.3) {};
			\node (pX) at (0.51,-0.7) {$\phi(X)$};
			\node (Y) at (4.5,-0.7) {$Y$};
			\node[fit=(aa)(bb)(cc)(dd)(pX)(pZ), draw] (boxXX) {};
			\draw[fill=black] (1.5,3.5) circle (3pt) node (phh){} node[above=3pt] (ph) {$\psi(h)$};
			\draw  (aa) -- (phh) node [above,pos=0.3,sloped] {$\scriptstyle{b_\bY}$}node [above,pos=0.5] (anc2){};
			\draw  (bb) -- (phh) node [above,pos = 0.3,sloped] {$\scriptstyle{(k,0)}$};
			\draw  (cc) -- (phh) node [above,pos = 0.3,sloped,rotate=180] {$\scriptstyle{d_j}$};
			\draw  (dd) -- (phh) node [above,pos = 0.35,sloped,rotate=180] {$\scriptstyle{(k,j)}$};
			\draw  (ee) -- (phh) node [above,pos = 0.35,sloped,rotate=180] {$\scriptstyle{b_\bY}$};
			\node[fit=(boxXX)(end), draw] (boxY) {};
			\node (anc2) at (0.5,2) {};
			\draw[-To,shorten <= 2pt,,shorten >= 2pt] (boxX) -- (boxY) node [above,pos=0.5] {$\phi$};
			\draw[-To] (anc1) -- (anc2) node [above,pos=0.5]{$\tilde\psi$};
			\draw[-To,shorten <= 1cm,shorten >= 1cm] (hh) -- (phh) node [above,pos=0.5]{$\psi$};
\end{scope}
	\end{tikzpicture}
	\caption{The construction of $\psi$}\label{figpsi}
	\end{figure}
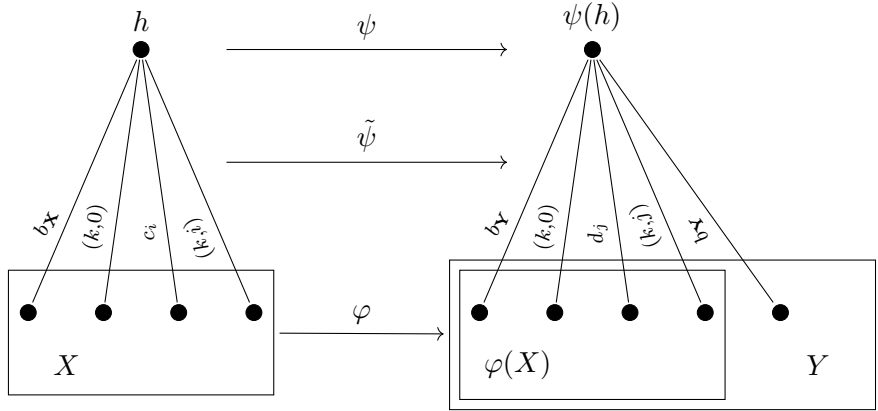
	Since $\tilde\psi$ and $\phi$ are injective, it follows that $\psi$ is injective, too. In order to see that $\psi\colon K(\bX)\to K(\bY)$ is indeed an embedding, according to Corollary~\ref{cor_sim_EMB}, we need to show that there exists $\hat\psi\colon E(K(\bX))\embedsto E(K(\bY))$ such that the following diagram commutes:
	\begin{equation}\label{psiemb}
		\begin{tikzcd}
			(C_\bX^{(X)}\dotcup X)^2\ar[r,"\eta_{K(\bX)}"]\ar[d,"\psi^2"'] & E(K(\bX))\ar[d,hook',"\hat\psi"]\\
			(C_\bY^{(Y)}\dotcup Y)^2\ar[r,"\eta_{K(\bY)}"] & E(K(\bY)).
		\end{tikzcd}
	\end{equation}
		To this end let us define $\hat\psi\coloneqq\zeta_\bY\circ \tilde\psi\circ\zeta_\bX^{-1}$. Now what remains is to show that the following diagram commutes:
		\begin{equation}\label{four}
		\begin{tikzcd}
			(C_\bX^{(X)}\dotcup X)^2\ar[rr,bend left,"\eta_{K(\bX)}"]\ar[d,hook',"\psi^2"']\ar[r,"\tilde\eta_\bX"] & C_\bX \ar[r,"\zeta_\bX","\cong"'] \ar[d,hook',"\tilde\psi"] & E(K(\bX))\ar[d,hook',"\hat\psi"]\\
			(C_\bY^{(Y)}\dotcup Y)^2\ar[rr,bend right, "\eta_{K(\bY)}"']\ar[r,"\tilde\eta_\bX"] & C_\bY \ar[r,"\zeta_\bY","\cong"']& E(K(\bY)).
		\end{tikzcd}
		\end{equation}
		Note that the upper and the lower triangle in \eqref{four} commute by \eqref{heart}. The right hand rectangle  of \eqref{four} commutes by construction. It remains to check that the left hand rectangle of \eqref{four} commutes. For this we take arbitrary but distinct $x,y\in X$ and $h_1,h_2\in C_\bX^{(X)}$, and chase them through this rectangle:
		\[
		\begin{tikzcd}[row sep=small]
			(x,x) \ar[ddd,mapsto,"\psi^2"'] \ar[r,mapsto,"\tilde\eta_\bX"]& \bot_\bX \ar[dd,mapsto,"\tilde\psi"]\\
			 \\
			 & \bot_\bY\ar[d,equals]\\
			 (\phi(x),\phi(x)) \ar[r,mapsto,"\tilde\eta_\bY"]& \bot_\bY
		\end{tikzcd}\qquad
		\begin{tikzcd}[row sep=small]
			(x,y) \ar[ddd,mapsto,"\psi^2"'] \ar[r,mapsto,"\tilde\eta_\bX"]& \eta_\bX(x,y)\ar[dd,mapsto,"\tilde\psi"]\\
			 \\
			 & \hat\phi(\eta_\bX(x,y))\ar[d,equals,"\text{by }\eqref{star}"]\\
			 (\phi(x),\phi(y)) \ar[r,mapsto,"\tilde\eta_\bY"]& \eta_\bY(\phi(x),\phi(y))
		\end{tikzcd}		
		\]
		\[
		\begin{tikzcd}[row sep=small]
			(h_1,x) \ar[ddd,mapsto,"\psi^2"'] \ar[r,mapsto,"\tilde\eta_\bX"]& h_1(x)\ar[dd,mapsto,"\tilde\psi"]\\
			 \\
			 & \tilde\psi(h_1(x))\ar[d,equals,"\text{by }\eqref{psi}"]\\
			 (\psi(h_1),\phi(x)) \ar[r,mapsto,"\tilde\eta_\bY"]& \psi(h_1)(\phi(x))
		\end{tikzcd}	
		\qquad	
		\begin{tikzcd}[row sep=small]
			(x,h_1) \ar[ddd,mapsto,"\psi^2"'] \ar[r,mapsto,"\tilde\eta_\bX"]& h_1(x)\ar[dd,mapsto,"\tilde\psi"]\\
			 \\
			 & \tilde\psi(h_1(x))\ar[d,equals,"\text{by }\eqref{psi}"]\\
			 (\phi(x),\psi(h_1)) \ar[r,mapsto,"\tilde\eta_\bY"]& \psi(h_1)(\phi(x))
		\end{tikzcd}	
		\]
		\[
		\begin{tikzcd}[row sep=small]
			(h_1,h_2) \ar[ddd,mapsto,"\psi^2"'] \ar[r,mapsto,"\tilde\eta_\bX"]& b_\bX \ar[dd,mapsto,"\tilde\psi"]\\
			 \\
			 & b_\bY \ar[d,equals]\\
			 (\psi(h_1),\psi(h_2)) \ar[r,mapsto,"\tilde\eta_\bY"]& b_\bY
		\end{tikzcd}	
		\qquad	
		\begin{tikzcd}[row sep=small]
			(h_1,h_1) \ar[ddd,mapsto,"\psi^2"'] \ar[r,mapsto,"\tilde\eta_\bX"]& \bot_\bX \ar[dd,mapsto,"\tilde\psi"]\\
			 \\
			 & \bot_\bY\ar[d,equals]\\
			 (\psi(h_1),\psi(h_1)) \ar[r,mapsto,"\tilde\eta_\bY"]& \bot_\bY.
		\end{tikzcd}
		\]
		Thus \eqref{four}, and,  in particular \eqref{psiemb}  commute. Consequently,  $\psi$ is an embedding of echeloned spaces. Let us define
		\[ K(\phi)\coloneqq \psi.\]
		Now that the action of $K$ on morphisms is defined, we still need to show that $K$ is indeed a functor, \ie
		\begin{enumerate}[label=(\roman*), ref=(\roman*)]
			\item \label{func1}$\forall\bX\in\calC\,:\, K(\id_\bX) = \id_{K(\bX)}$, and
			\item \label{func2} $\forall \phi_1\colon\bX\embedsto\bY,\, \phi_2\colon\bY\embedsto\bZ\,:\, K(\phi_2\circ\phi_1) = K(\phi_2)\circ K(\phi_1)$.
		\end{enumerate}
		
		\textbf{About \ref{func1}:} For some finite echeloned space $\bX$ consider $K(\id_\bX)\colon C_\bX^{(X)}\dotcup X\to C_\bX^{(X)}\dotcup X$. Clearly, for each $x\in X$ we have $K(\id_\bX)(x)= x$. So let $h\in C_\bX^{(X)}$. For simplicity of notation, let us denote $K(\id_\bX)$ by $\psi$. Then $(K(\id_\bX)(h))(x)= \tilde\psi(h(x))$, for each $x\in X$. 
		\begin{itemize}
			\item If $h(x)=b_\bX$, then $\tilde\psi(h(x)) = \tilde\psi(b_\bX) = b_\bX$;
			\item If $h(x)=(k,0)$, then $\tilde\psi(h(x))= \tilde\psi((k,0)) = (k,0)$;
			\item If $h(x)=c_i$, then $\tilde\psi(h(x)) = \tilde\psi(c_i) = \widehat{\id_\bX}(c_i) = c_i$;
			\item If $h(x) = (k,i)$, then $\tilde\psi(h(x)) = \tilde\psi((k,i)) = (k,i)$, because $\widehat{\id_\bX}(c_i) = c_i$;
		\end{itemize} 
		Thus, $K(\id_\bX) = \id_{K(\bX)}$. 
		
		\textbf{About \ref{func2}:} Given 
		\[
		\begin{tikzcd}
			\bX \ar[r,hook,"\phi_1"] & \bY\ar[r,hook,"\phi_2"] & \bZ
		\end{tikzcd}
		\]
		 Let us denote $\phi\coloneqq \phi_2\circ\phi_1$, $\psi\coloneqq K(\phi)$, $\psi_1 = K(\phi_1)$, and $\psi_2\coloneqq K(\phi_2)$. Then
		\[
		\begin{tikzcd}
			C_\bX^{(X)}\dotcup X \ar[r,hook,"\psi_1"] & C_\bY^{(Y)}\dotcup Y\ar[r,hook,"\psi_2"] & C_\bZ^{(Z)}\dotcup Z,\text{ and}
		\end{tikzcd}
		\]
		\[
		\begin{tikzcd}
			C_\bX \ar[r,hook,"\tilde\psi_1"] & C_\bY\ar[r,hook,"\tilde\psi_2"] & C_\bZ.
		\end{tikzcd}
		\]
		Suppose $E(\bX)=\{\bot_\bX,c_1,\dots,c_n\}$, $E(\bY) = \{\bot_\bY,d_1,\dots,d_m\}$, and $E(\bZ)=\{\bot_\bZ,e_1,\dots,e_s\}$, where
        \begin{multline*}
        \bot_\bX <_{E(\bX)} c_1 <_{E(\bX)} \dots <_{E(\bX)} c_n,\quad \bot_\bY <_{E(\bY)} d_1 <_{E(\bY)} \dots <_{E(\bY)} d_m, \\\text{and }
        \bot_\bZ <_{E(\bZ)} e_1 <_{E(\bZ)} \dots <_{E(\bZ)} e_s.
        \end{multline*}
  Then for every $x\in X$ we compute:
		\[
		\begin{tikzcd}
			x \ar[d,mapsto,"\psi_1"']\ar[r,mapsto,"\psi"]& \phi(x)\arrow[d,equals]\\
			\phi_1(x) \ar[r,mapsto,"\psi_2"]& \phi_2(\phi_1(x)).
		\end{tikzcd}
		\]
		So let $h\in C_\bX^{(X)}$. Then 
		\[
		\psi(h)\colon z\mapsto \begin{cases}
			\tilde\psi(h(x)) & \text{if } z=\phi(x) = \phi_2(\phi_1(x)),\\
			b_\bZ & \text{else},
		\end{cases}
		\]
		where
		\[
		\tilde\psi(h(x)) = \begin{cases}
			b_\bZ & \text{if } h(x)=b_\bX,\\
			(k,0) & \text{if } h(x) = (k,0),\\
			\hat\phi(h(x)) & \text{if } h(x)\in E(\bX),\\
			(k,j) & \text{if } h(x)=(k,i),\,i>0,\,\hat\phi(c_i) = e_j.
		\end{cases}
		\]Note that $\hat\phi = \hat\phi_2\circ\hat\phi_1$ (this follows from the uniqueness of $\hat\phi$ for $\phi$).
		
		On the other hand, 
		\[
		(K(\phi_2)\circ K(\phi_1))(h) = (\psi_2\circ\psi_1)(h)\in C_\bZ^{(Z)},
		\]
		and
		\[
		(\psi_2\circ\psi_1)(h)\colon z\mapsto \begin{cases}
			\tilde\psi_2((\psi_1(h))(y)) &\text{if }  z=\phi_2(y),\, y\in Y\\
			b_\bZ & \text{else.}
		\end{cases}
		\]
		Thus
		\[
		(\psi_2\circ\psi_1)(h)\colon z\mapsto\begin{cases}
			\tilde\psi_2(\tilde\psi_1(h(x))) & \text{if } z=\phi_2(\phi_1(x)),\, x\in X,\\
			\tilde\psi_2(b_\bY)=b_\bZ & \text{if } z=\phi_2(y),\, y\in Y\setminus\phi_1(X),\\
			b_\bZ & \text{else.}
		\end{cases}
		\]
		Figure~\ref{figcomp} illustrates the construction of the action of $K(\phi_2)\circ K(\phi_1)$. 
	\begin{figure}
	\begin{tikzpicture}[mycontainer/.style={draw=gray, inner sep=1ex},xscale=0.85]
		\begin{scope}
			\draw[fill=black] %
				(0,0) circle (3pt) node (a) {}
				(1,0) circle (3pt) node (b) {}
				(2,0) circle (3pt) node (c) {}
				(3,0) circle (3pt) node (d) {};
			\node (Z) at (2,0.3) {};
			\node (X) at (0.51,-0.7) {$X$};
			\node[fit=(a)(b)(c)(d)(X)(Z), draw] (boxX) {};
			\draw[fill=black] (1.5,3.5) circle (3pt) node (hh){} node[above=3pt] (h) {$h$};
			\draw  (a) -- (hh) node [above,pos=0.3,sloped] {$\scriptstyle{b_\bX}$};
			\draw  (b) -- (hh) node [above,pos = 0.3,sloped] {$\scriptstyle{(k,0)}$};
			\draw  (c) -- (hh) node [above,pos = 0.3,sloped,rotate=180] {$\scriptstyle{c_i}$};
			\draw  (d) -- (hh) node [above,pos = 0.3,sloped,rotate=180] {$\scriptstyle{(k,i)}$};
			\node (anc1) at (2.5,1.5) {};
		\end{scope}
		\begin{scope}[xshift=50mm]
			\draw[fill=black] %
				(0,0) circle (3pt) node (aa) {}
				(1,0) circle (3pt) node (bb) {}
				(2,0) circle (3pt) node (cc) {}
				(3,0) circle (3pt) node (dd) {}
				(4,0) circle (3pt) node (ee) {};
			\node (end) at (4,0) {};
			\node (pZ) at (2,0.3) {};
			\node (pX) at (0.51,-0.7) {$\phi_1(X)$};
			\node (Y) at (4.0,-0.7) {$Y$};
			\node at (4.0,2.8){$\hat\phi_1(c_i)=d_\ell$};
			\node[fit=(aa)(bb)(cc)(dd)(pX)(pZ), draw] (boxXX) {};
			\draw[fill=black] (1.5,3.5) circle (3pt) node (phh){} node[above=3pt] (ph) {$\psi_1(h)$};
			\draw  (aa) -- (phh) node [above,pos=0.3,sloped] {$\scriptstyle{b_\bY}$}node [above,pos=0.5] (anc2){};
			\draw  (bb) -- (phh) node [above,pos = 0.3,sloped] {$\scriptstyle{(k,0)}$};
			\draw  (cc) -- (phh) node [above,pos = 0.3,sloped,rotate=180] {$\scriptstyle{d_\ell}$};
			\draw  (dd) -- (phh) node [above,pos = 0.35,sloped,rotate=180] {$\scriptstyle{(k,\ell)}$};
			\draw  (ee) -- (phh) node [above,pos = 0.35,sloped,rotate=180] {$\scriptstyle{b_\bY}$};
			\node[fit=(boxXX)(end), draw] (boxY) {};
			\node (anc2) at (0,1.5) {};
			\draw[-To,shorten <= 2pt,,shorten >= 2pt] (boxX) -- (boxY) node [above,pos=0.5] {$\phi_1$};
			\draw[-To] (anc1) -- (anc2) node [above,pos=0.5]{$\tilde\psi_1$};
			\draw[-To,shorten <= 1cm,shorten >= 1cm] (hh) -- (phh) node [above,pos=0.5]{$\psi_1$};
			\node (anc3) at (3.5,1.5) {};
		\end{scope}
		\begin{scope}[xshift=115mm]
			\draw[fill=black] %
				(0,0) circle (3pt) node (aaa) {}
				(1,0) circle (3pt) node (bbb) {}
				(2,0) circle (3pt) node (ccc) {}
				(3,0) circle (3pt) node (ddd) {}
				(4,0) circle (3pt) node (eee) {}
				(5.5,0) circle (3pt) node (fff) {};
			\node (eend) at (4.5,0) {}; %(end) 
			\node (eeend) at (5.5,0) {}; %(end) 
			\node (ppZ) at (2,0.3) {};
			\node (ppX) at (0.51,-0.7) {$\phi_2(\phi_1(X))$};
			\node (pY) at (4.0,-0.7) {$\phi_2(Y)$};
			\node (pY) at (5.5,-0.7) {$Z$};
			\node at (4.5,2.8){$\hat\phi_2(d_\ell)=e_j$};
			\node[fit=(aaa)(bbb)(ccc)(ddd)(ppX)(ppZ), draw] (bboxXX) {};
			\draw[fill=black] (2.0,3.5) circle (3pt) node (pphh){} node[above=3pt] (ph) {$\psi_2((\psi_1(h))$};
			\draw  (aaa) -- (pphh) node [above,pos=0.3,sloped] {$\scriptstyle{b_\bZ}$}node [above,pos=0.5] (anc2){};
			\draw  (bbb) -- (pphh) node [above,pos = 0.3,sloped] {$\scriptstyle{(k,0)}$};
			\draw  (ccc) -- (pphh) node [above,pos = 0.3,sloped,rotate=0] {$\scriptstyle{e_j}$};
			\draw  (ddd) -- (pphh) node [above,pos = 0.35,sloped,rotate=180] {$\scriptstyle{(k,j)}$};
			\draw  (eee) -- (pphh) node [above,pos = 0.35,sloped,rotate=180] {$\scriptstyle{b_\bZ}$};
			\draw  (fff) -- (pphh) node [above,pos = 0.35,sloped,rotate=180] {$\scriptstyle{b_\bZ}$};
			\node[fit=(bboxXX)(eend), draw] (bboxY) {};
			\node[fit=(bboxY)(eeend), draw] (bboxY) {};
			\node (anc4) at (0,1.5) {};
			\draw[-To,shorten <= 2pt,,shorten >= 2pt] (boxY) -- (bboxY) node [above,pos=0.5] {$\phi_2$};
			\draw[-To] (anc3) -- (anc4) node [above,pos=0.5]{$\tilde\psi_2$};
			\draw[-To,shorten <= 1cm,shorten >= 1cm] (phh) -- (pphh) node [above,pos=0.5]{$\psi_2$};
		\end{scope}
	\end{tikzpicture}
	\caption{The composition $K(\phi_2)\circ K(\phi_1)$}\label{figcomp}
	\end{figure}
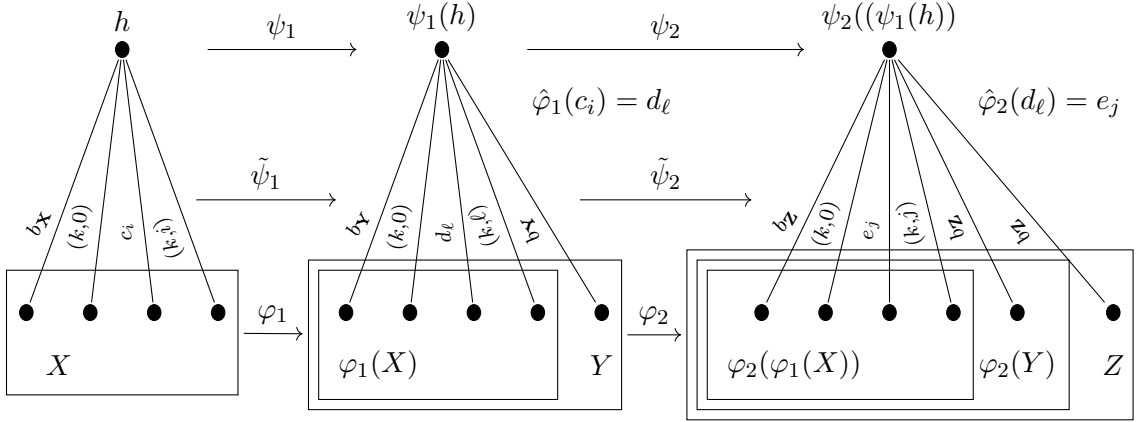
	It remains to check that $\tilde\psi=\tilde\psi_2\circ\tilde\psi_1$. As usual, we do so by considering possible cases separately.  During these computations, assume that for $c_i \in E(\bX)\setminus\{\bot_\bX\}$, we have  that $\hat\phi_1(c_i)=d_\ell$ and $\hat\phi_2(d_\ell)=e_j$. Then 
	\begin{gather*}
		\tilde\psi_2(\tilde\psi_1(\bot_\bX))  = \tilde\psi_2(\bot_\bY) = \bot_\bZ = \tilde\psi(\bot_\bX),\\
		\tilde\psi_2(\tilde\psi_1(b_\bX)) = \tilde\psi_2(b_\bY) = b_\bZ = \tilde\psi(b_\bX),\\
		\tilde\psi_2(\tilde\psi_1((k,0))) = \tilde\psi_2((k,0)) = (k,0) = \tilde\psi((k,0)),\\
		\tilde\psi_2(\tilde\psi_1(c_i)) = \tilde\psi_2(\hat\phi_1(c_i)) = \hat\phi_2(\hat\phi_1(c_i)) = \hat\phi(c_i) = \tilde\psi(c_i),\\
		\tilde\psi_2(\tilde\psi_1((k,i))) = \tilde\psi_2((k,\ell)) = (k,j) = \tilde\psi((k,i)).
	\end{gather*}
	This finishes the proof that $K$ is a functor.
	\begin{PROP}
		$K\colon\cat{A}\to\cat{C}$ is a \Katetov functor.
	\end{PROP}
	\begin{proof}
		For every finite echeloned space $\bX$ let $\lambda_\bX\colon \bX\embedsto K(\bX)$ be the identical embedding. It is not hard to check that this defines a natural transformation $\lambda\colon\Id\embedsto K$.
		
		Let now $\bX$ be a finite echeloned space and let $\bY$ be a one-point extension  of $\bX$ (\ie $Y=X\dotcup\{y\}$ and the identical embedding of $X$ into $Y$ is an embedding of echeloned spaces). Denote the identical embedding of $\bX$ into $\bY$ by $e$. Suppose $E(\bX)=\{\bot_\bX,c_1,\dots,c_n\}$, where
		\[
		\bot_\bX<_{E(\bX)} c_1 <_{E(\bX)} c_2 <_{E(\bX)}\dots <_{E(\bX)} c_n.
		\] 
		Then $E(\bY)$ is of the shape:
		\begin{multline*}
		\bot_\bY <_{E(\bY)} d_{1,0} <_{E(\bY)} \dots d_{i_0,0} <_{E(\bY)} \hat{e}(c_1) <_{E(\bY)} d_{1,1} <_{E(\bY)} \dots \\ 
		\dots<_{E(\bY)} d_{i_1,1} <_{E(\bY)} \hat{e}(c_2)<_{E(\bY)}\dots <_{E(\bY)} \hat{e}(c_n) <_{E(\bY)} d_{1,n} <_{E(\bY)} \dots <_{E(\bY)} d_{i_n,n} ,
		\end{multline*}
		where $i_0+i_1+\dots+i_n\leq|X|$. Define
      \[ 
        h\colon X\to C_\bX\setminus\{\bot_\bX\},\quad
		 x\mapsto\begin{cases}
			(k,j) & \text{if } \eta_\bY(y,e(x)) = d_{k,j}, \text{ where }j=0,\dots,n,\, k=1,\dots, i_j,\\
			c_j & \text{if } \eta_\bY(y,e(x)) = \hat{e}(c_j),\text{ where } j=1,\dots,n.
		\end{cases}
		\]
		Finally define 
        \[
        g\colon Y\to X\dotcup C_\bX^{(X)},\quad  
		x\mapsto\begin{cases}
			x & \text{if } x\in X,\\
			h & \text{if } x=y.
		\end{cases}
		\]
		By construction, $g$ is injective. Next we show that $g\colon\bY\to K(\bX)$ is an embedding. Let $\gamma\colon E(\bY)\to C_\bX$  be given by
		\[
		\gamma\colon \hat{e}(c_j)\mapsto c_j,\quad \bot_\bY\mapsto\bot_\bX,\quad d_{k,j}\mapsto (k,j)\quad \text{ for all } j\in\{0,\dots,n\},\, k\in \{1,\dots,i_j\}.
		\]
		Clearly, $\gamma$ is an order embedding and the following diagram commutes:
		\[
		\begin{tikzcd}
			X^2 \ar[r,"\eta_\bX"] \ar[d,hook',"e^2"'] & E(\bX)\ar[r,hook,"="]\ar[d,hook',"\hat{e}"] & C_\bX \ar[d,equals]\\
			Y^2\ar[r,"\eta_\bY"] & E(\bY)\ar[r,hook,"\gamma"] & C_\bX.
		\end{tikzcd}
		\]
		Next we show that the following diagram commutes, too:
		\begin{equation}\label{fourstar}
		\begin{tikzcd}
			Y^2  \ar[r,"\eta_\bY"] \ar[d,hook',"g^2"']& E(\bY)\ar[d,hook',"\gamma"] \\
			(C_\bX^{(X)}\dotcup X)^2 \ar[r,"\tilde\eta_\bX"]& C_\bX.
		\end{tikzcd}
		\end{equation}
		To this end let $(x,z)\in Y^2$. We distinguish six cases:
		\begin{enumerate}[label=\textbf{Case \arabic*:}, ref=\arabic*,nosep,align=left,leftmargin=0em,labelindent=0em,itemindent=1em,labelsep=0.5em,labelwidth=!]
		\item Suppose that $x=z\neq y$. Then
		\[
		\begin{tikzcd}[row sep=small]
		(x,z) \ar[r,mapsto,"\eta_\bY"]\ar[ddd,mapsto,"g^2"'] & \bot_\bY\ar[dd,mapsto,"\gamma"]\\
		\\
		&  \bot_\bX\ar[d,equals]\\
		(x,z) \ar[r,mapsto,"\tilde\eta_\bX"]& \bot_\bX.
		\end{tikzcd}
		\]
		\item Suppose that $x=z=y$. Then
		\[
		\begin{tikzcd}[row sep=small]
		(x,z) \ar[r,mapsto,"\eta_\bY"]\ar[ddd,mapsto,"g^2"'] & \bot_\bY\ar[dd,mapsto,"\gamma"]\\
		\\
		&  \bot_\bX\ar[d,equals]\\
		(h,h) \ar[r,mapsto,"\tilde\eta_\bX"]& \bot_\bX.
		\end{tikzcd}
		\]
		\item Suppose that $x\neq z$ and that $x,z\in X$. Then $\eta_\bX(x,z)=c_i$ for some $i$ and: 
		\[
		\begin{tikzcd}[row sep=small]
		(x,z) \ar[r,mapsto,"\eta_\bY"]\ar[ddd,mapsto,"g^2"'] & \hat{e}(c_i)\ar[dd,mapsto,"\gamma"]\\
		\\
		&  c_i\ar[d,equals]\\
		(x,z) \ar[r,mapsto,"\tilde\eta_\bX"]& c_i.
		\end{tikzcd}
		\]
		\item \label{cc4} Suppose that $x\neq z$ and $x=y$, and that $\eta_\bY(y,e(z))= d_{k,j}$ for some $j\in\{0,\dots,n\}$, and $k\in\{1,\dots,i_j\}$. Then 
		\[
		\begin{tikzcd}[row sep=small]
		(x,z) \ar[r,mapsto,"\eta_\bY"]\ar[ddd,mapsto,"g^2"'] & d_{k,j}\ar[dd,mapsto,"\gamma"]\\
		\\
		&  (k,j)\ar[d,equals]\\
		(h,z) \ar[r,mapsto,"\tilde\eta_\bX"]& (k,j).
		\end{tikzcd}
		\]
		\item\label{cc5} Suppose that $x\neq z$ and $x=y$, and that $\eta_\bY(x,e(z))= \hat{e}(c_j)$  for some $j\in\{0,\dots,n\}$. Then
		\[
		\begin{tikzcd}[row sep=small]
		(x,z) \ar[r,mapsto,"\eta_\bY"]\ar[ddd,mapsto,"g^2"'] & \hat{e}(c_j)\ar[dd,mapsto,"\gamma"]\\
		\\
		&  c_j\ar[d,equals]\\
		(h,z) \ar[r,mapsto,"\tilde\eta_\bX"]& c_j.
		\end{tikzcd}
		\]
		\item If $x\neq z$ and $z=y$, then we proceed analogously as in cases \ref{cc4} and \ref{cc5}.
		\end{enumerate}

		At this point we may conclude that \eqref{fourstar} commutes. Combining this with \eqref{heart}, we obtain  that  the following diagram commutes, too:
		\[
		\begin{tikzcd}
			Y^2 \ar[r,"\eta_\bY"]\ar[d,hook',"g^2"']& E(\bY)\ar[d,hook',"\gamma"]\\
			(C_\bX^{(X)}\dotcup X)^2 \ar[r,"\tilde\eta_\bX"]\ar[dr,"\eta_{K(\bX)}"']& C_\bX\ar[d,"\zeta_\bX","\cong"'] \\
			& E(K(\bX)).
		\end{tikzcd}
		\]
		Hence, by Corollary~\ref{cor_sim_EMB}, $g$ is an embedding with $\hat{g}=\zeta_\bX\circ\gamma$. 
		
		Finally, by the definition of $g$, the following diagram commutes:
		\[
		\begin{tikzcd}
			\bX \ar[d,hook',"="']\ar[dr,hook',"\lambda_\bX","="']\\
			\bY\ar[r,hook,"g"'] & K(\bX).
		\end{tikzcd}
		\]
		Thus, indeed, $K$ is a \Katetov functor.
	\end{proof}

\section{An alternative approach to echeloned spaces}%\label{sec_alternative}

As established in Remark \ref{rem_echelon} an echeloned space may be perceived as a set of points accompanied by a specific $4$-ary relation. The following definition records yet another, equivalent, approach to these structures.

\begin{DEF}\label{altDEF_sim}
An \emph{echeloned map} on a set $X$ is a surjective function $f \colon X^2 \twoheadrightarrow (C, \leq)$ for which:
\begin{enumerate}[label=(\roman*), ref=(\roman*)]
  \item $(C, \leq)$ is a linear order with minimum $\bot_C$,
  \item $f^{-1}(\bot_C) = \Delta_X$, and
  \item $f(x, y) = f(y, x)$, for all $x, y \in X$.
\end{enumerate}
We refer to $X$ as the \emph{set of points} of $f$.
\end{DEF}

\begin{REM} An echeloned map on a set $X$ is \emph{finite} if $X$ is finite. 
\end{REM}

\begin{DEF} Let $f \colon X^2 \twoheadrightarrow (C, \leq)$ and $g \colon Y^2 \twoheadrightarrow (D, \leq)$ be two echeloned maps on $X$ and $Y$ respectively.
We say that a function $\chi\colon X\to Y$ is a \emph{homomorphism} from $f$ to $g$, and write $\chi\colon f\to g$,  if there
exists an order-preserving map $\bar{\chi} \colon (C, \leq) \rightarrow (D, \leq)$ such that the diagram below commutes:
\[\begin{tikzcd}
X^2 \arrow[r, twoheadrightarrow, "f"]\arrow[d, rightarrow, "\chi^2"{name=L, left}] & (C, \leq)\arrow[d, rightarrow, "\bar{\chi}"] \\
Y^2 \arrow[r, twoheadrightarrow, "g"]& (D, \leq).
\end{tikzcd}
\]
We call $\chi$ an \emph{embedding} if $\chi$ is injective and $\bar{\chi}$ is an order embedding.
\end{DEF}

\begin{REM}
If $\bar{\chi}$ in the above definition exists, then it is unique.
\end{REM}

\begin{OBS}
We describe the connection between echeloned spaces and Definition~\ref{altDEF_sim} in some detail.

Let $\cat{A}$ be the category of echeloned spaces with homomorphisms, and let $\cat{B}$ be the category of echeloned maps with homomorphisms. Clearly, if $\bX = (X, \leq_\bX)$ is an echeloned space, then $\eta_\bX$ is an echeloned map. Moreover, the homomorphisms between two echeloned spaces are the same as those between their corresponding echeloned maps (see Lemma~\ref{lem_sim_HOM}).  This means that $F \colon \cat{A} \to \cat{B},\, \bX \mapsto \eta_{\bX},\, h \mapsto h$ is a well-defined functor.

Now, define a functor $G \colon \cat{B} \to \cat{A}$. To every echeloned map $f \colon X^2 
\twoheadrightarrow (C, \leq)$ we associate an echeloned space $\bX_f = (X, \leq_\bX)$ according to the following rule:
\[(x,y) \leq_\bX (u,v) \quad :\Longleftrightarrow \quad f(x,y) \leq f(u, v).\]
Again, by Lemma~\ref{lem_sim_HOM},  $h \colon f \to g$ is a homomorphism of echeloned maps, then $h \colon \bX_f \to \bX_g$ is a homomorphism, too. Hence the assignment $G\colon\cat{B}\to\cat{A},\, f \mapsto \bX_f,\, h \mapsto h$ is a well-defined functor. Note that $G\circ F = \Id_{\cat{A}}$. Conversely, for each echeloned map $f \colon X^2 \twoheadrightarrow (C, \leq)$ note that the identity function $\id_X$ is an isomorphism from $F(G(f))$ to $f$. Moreover, $(\phi_f)_{f\in\ob(\cat{B})}\colon F\circ G\to \Id_{\cat{B}}$ with $\phi_f \colon F(G(f)) \to f, \phi_f \coloneqq  \id_X$ is a natural isomorphism.  This shows that  $\cat{A}$ and $\cat{B}$ are equivalent categories. Note that $F$ and $G$,
both, preserve finiteness and embeddings.
\end{OBS}

All the proofs from the previous sections can be rewritten easily in this alternative approach. 

\section*{Acknowledgements}
The authors thank Dragan \Masulovic for enlightening discussions on \Katetov functors and Colin Jahel for helpful discussions concerning the Ramsey property.

\bibliographystyle{habbrv}
\bibliography{SimStr}

\begin{thebibliography}{10}
\expandafter\ifx\csname url\endcsname\relax
  \def\url#1{\texttt{#1}}\fi
\expandafter\ifx\csname doi\endcsname\relax
  \def\doi#1{\burlalt{doi:#1}{http://dx.doi.org/#1}}\fi
\expandafter\ifx\csname urlprefix\endcsname\relax\def\urlprefix{URL }\fi
\expandafter\ifx\csname href\endcsname\relax
  \def\href#1#2{#2}\fi
\expandafter\ifx\csname burlalt\endcsname\relax
  \def\burlalt#1#2{\href{#2}{#1}}\fi

\bibitem{BeckerKechris}
H.~Becker and A.~S. Kechris.
\newblock {\em The descriptive set theory of {P}olish group actions}, volume 232 of {\em London Mathematical Society Lecture Note Series}.
\newblock Cambridge University Press, Cambridge, 1996.
\newblock \doi{10.1017/CBO9780511735264}.

\bibitem{birkhoff1967lattice}
G.~Birkhoff.
\newblock {\em Lattice {T}heory}.
\newblock American Mathematical Society Colloquium Publications, Vol. XXV. American Mathematical Society, Providence, R.I., third edition, 1967.

\bibitem{Brown-1971}
L.~G. Brown.
\newblock Note on the open mapping theorem.
\newblock {\em Pacific J. Math.}, 38:25--28, 1971.
\newblock \doi{10.2140/pjm.1971.38.25}.

\bibitem{Brown-1972}
L.~G. Brown.
\newblock Topologically complete groups.
\newblock {\em Proc. Amer. Math. Soc.}, 35:593--600, 1972.
\newblock \doi{10.2307/2037655}.

\bibitem{Cam1990}
P.~J. Cameron.
\newblock {\em Oligomorphic permutation groups}, volume 152 of {\em London Mathematical Society Lecture Note Series}.
\newblock Cambridge University Press, Cambridge, 1990.
\newblock \doi{10.1017/CBO9780511549809}.

\bibitem{Chv04}
V.~Chv{\'a}tal.
\newblock Sylvester-{G}allai theorem and metric betweenness.
\newblock {\em Discrete and Computational Geometry}, 31(2):175--195, 2004.
\newblock \doi{10.1007/s00454-003-0795-6}.

\bibitem{Fre05}
M.~R. Fr{\'e}chet.
\newblock La notion d'{\'e}cart dans le calcul fonctionnel.
\newblock {\em C. R. Acad. Sci., Paris}, 140:772--774, 1905.

\bibitem{Fre06}
M.~R. Fr{\'e}chet.
\newblock Sur quelques points du calcul fonctionnel.
\newblock {\em Rend. Circ. Mat. Palermo}, 22:1--74, 1906.
\newblock \doi{10.1007/BF03018603}.

\bibitem{Hau14}
F.~Hausdorff.
\newblock {\em Grundz{\"u}ge der {M}engenlehre}.
\newblock Verlag von Veit \& Comp., Leipzig, 1914.

\bibitem{hodges1993model}
W.~Hodges.
\newblock {\em Model theory}, volume~42 of {\em Encyclopedia of Mathematics and its Applications}.
\newblock Cambridge University Press, Cambridge, 1993.
\newblock \doi{10.1017/CBO9780511551574}.

\bibitem{Hu49}
S.-T. Hu.
\newblock Boundedness in a topological space.
\newblock {\em J. Math. Pures Appl. (9)}, 28:287--320, 1949.

\bibitem{hubivcka2019all}
J.~Hubi{\v{c}}ka and J.~Ne{\v{s}}et{\v{r}}il.
\newblock All those {R}amsey classes ({R}amsey classes with closures and forbidden homomorphisms).
\newblock {\em Advances in Mathematics}, 356:106791, 2019.
\newblock \doi{10.1016/j.aim.2019.106791}.

\bibitem{kechris2003fraisse}
A.~S. Kechris, V.~G. Pestov, and S.~Todor\v{c}evi\'c.
\newblock Fra\"{\i}ss\'{e} limits, {R}amsey theory, and topological dynamics of automorphism groups.
\newblock {\em Geom. Funct. Anal.}, 15(1):106--189, 2005.
\newblock \doi{10.1007/s00039-005-0503-1}.

\bibitem{KubMas16}
W.~Kubi{\'s} and D.~Ma{\v s}ulovi{\'c}.
\newblock Kat{\v e}tov functors.
\newblock {\em Applied Categorical Structures}, pages 1--34, 2016.
\newblock \doi{10.1007/s10485-016-9461-z}.

\bibitem{Men28}
K.~Menger.
\newblock Untersuchungen {\"u}ber allgemeine {M}etrik.
\newblock {\em Mathematische Annalen}, 100(1):75--163, Dec. 1928.
\newblock \doi{10.1007/bf01448840}.

\bibitem{Nes05}
J.~Ne{\v s}et{\v r}il.
\newblock Ramsey classes and homogeneous structures.
\newblock {\em Combin. Probab. Comput.}, 14(1-2):171--189, 2005.
\newblock \doi{10.1017/S0963548304006716}.

\bibitem{pestov1998free}
V.~G. Pestov.
\newblock On free actions, minimal flows, and a problem by {E}llis.
\newblock {\em Transactions of the American Mathematical Society}, 350(10):4149--4165, 1998.
\newblock \doi{10.1090/S0002-9947-98-02329-0}.

\bibitem{PestovBook}
V.~G. Pestov.
\newblock {\em Dynamics of infinite-dimensional groups. The Ramsey-Dvoretzky-Milman phenomenon}, volume~40 of {\em University Lecture Series}.
\newblock American Mathematical Society, Providence, RI, 2006.
\newblock \doi{10.1090/ulect/040}.

\bibitem{Pes19}
V.~G. Pestov.
\newblock {\em Elementos da teoria de aprendizagem de m\'{a}quina supervisionada}.
\newblock 32$^{\rm o}$ Col\'{o}quio Brasileiro de Matem\'{a}tica. Instituto Nacional de Matem\'{a}tica Pura e Aplicada (IMPA), Rio de Janeiro, 2019.

\bibitem{Roe03}
J.~Roe.
\newblock {\em Lectures on coarse geometry}, volume~31 of {\em University Lecture Series}.
\newblock American Mathematical Society, Providence, RI, 2003.
\newblock \doi{10.1090/ulect/031}.

\bibitem{RoelckeDierolf}
W.~Roelcke and S.~Dierolf.
\newblock {\em Uniform structures on topological groups and their quotients}.
\newblock Advanced Book Program. McGraw-Hill International Book Co., New York, 1981.

\bibitem{Sauer}
N.~W. Sauer.
\newblock Distance sets of {U}rysohn metric spaces.
\newblock {\em Canad. J. Math.}, 65(1):222--240, 2013.
\newblock \doi{10.4153/CJM-2012-022-4}.

\bibitem{tarzi2014multicoloured}
S.~Tarzi.
\newblock Multicoloured random graphs: Constructions and symmetry.
\newblock {\em ArXiv e-prints}, 2014, \burlalt{1406.7870}{http://arxiv.org/abs/1406.7870}.

\bibitem{truss1985group}
J.~K. Truss.
\newblock The group of the countable universal graph.
\newblock {\em Math. Proc. Cambridge Philos. Soc.}, 98(2):213--245, 1985.
\newblock \doi{10.1017/S0305004100063428}.

\end{thebibliography}
\end{document}